\documentclass[a4paper]{amsart}
\usepackage{amssymb, enumerate}
\usepackage{hyperref, aliascnt, graphicx,xcolor}

\usepackage[capitalise, nameinlink, noabbrev, nosort]{cleveref} 

\theoremstyle{plain}
\newtheorem{lma}{Lemma}[section]
\crefname{lma}{Lemma}{Lemmata}
\newtheorem{thm}[lma]{Theorem}
\crefname{thm}{Theorem}{Theorems}
\newtheorem{cor}[lma]{Corollary}
\crefname{cor}{Corollary}{Corollaries}
\newtheorem{prp}[lma]{Proposition}
\crefname{prp}{Proposition}{Propositions}

\theoremstyle{definition}
\newtheorem{pgr}[lma]{}
\crefname{pgr}{Paragraph}{Paragraphs}
\newtheorem{dfn}[lma]{Definition}
\crefname{dfn}{Definition}{Definitions}

\theoremstyle{remark}
\newtheorem{rmk}[lma]{Remark}
\crefname{rmk}{Remark}{Remarks}
\newtheorem{rmks}[lma]{Remarks}
\crefname{rmks}{Remarks}{Remarks}
\newtheorem{exa}[lma]{Example}
\crefname{exa}{Example}{Examples}
\newtheorem{qst}[lma]{Question}
\crefname{qst}{Question}{Questions}

\theoremstyle{plain}

\newcounter{IntroCount}

\newtheorem{thmIntro}[IntroCount]{Theorem}
\crefname{thmIntro}{Theorem}{Theorems}
\newtheorem{qstIntro}[IntroCount]{Question}
\crefname{qstIntro}{Question}{Questions}

\crefname{corIntro}{Corollary}{Corollaries}

\def\today{\number\day\space\ifcase\month\or   January\or February\or
   March\or April\or May\or June\or   July\or August\or September\or
   October\or November\or December\fi\   \number\year}

\newcommand{\andSep}{\,\,\,\text{ and }\,\,\,}
\newcommand{\axiomO}[1]{(O#1)}
\newcommand{\CuSgp}{$\mathrm{Cu}$-sem\-i\-group}

\newcommand{\NN}{{\mathbb{N}}}
\newcommand{\Cpct}{{\mathcal{K}}}
\newcommand{\red}{{\mathrm{red}}}
\newcommand{\ca}{$C^*$-algebra}
\newcommand{\dimnuc}{\dim_{\mathrm{nuc}}}

\DeclareMathOperator{\Cu}{Cu}
\DeclareMathOperator{\Lsc}{Lsc}

\title{Pure C*-algebras}
\date{\today}

\author{Ramon Antoine}
\author{Francesc Perera}
\author{Hannes Thiel}
\author{Eduard Vilalta}

\address{
R.~Antoine, 
Departament de Matem\`{a}tiques,
Universitat Aut\`{o}noma de Barcelona,
08193 Bellaterra, Barcelona, Spain, and
Centre de Recerca Matem\`atica, Edifici Cc, Campus de Bellaterra, 08193 Cerdanyola del Vall\`es, Barcelona, Spain}
\email[]{ramon.antoine@uab.cat}

\address{
F.~Perera, 
Departament de Matem\`{a}tiques,
Universitat Aut\`{o}noma de Barcelona,
\linebreak 08193 Bellaterra, Barcelona, Spain, and
Centre de Recerca Matem\`atica, Edifici Cc, Campus de Bellaterra,  08193 Cerdanyola del Vall\`es, Barcelona, Spain}
\email[]{francesc.perera@uab.cat}
\urladdr{https://mat.uab.cat/web/perera}

\address{Hannes~Thiel, 
Department of Mathematical Sciences, Chalmers University of Technology and the University of
Gothenburg, SE-412 96 Gothenburg, Sweden}
\email{hannes.thiel@chalmers.se}
\urladdr{www.hannesthiel.org}

\address{Eduard Vilalta, 
Department of Mathematical Sciences, Chalmers University of Technology and University of Gothenburg, SE-412 96 Gothenburg, Sweden}
\email[]{vilalta@chalmers.se}
\urladdr{www.eduardvilalta.com}

\thanks{
RA, FP and EV were partially supported by the Spanish Research State Agency (grant No.\  PID2020-113047GB-I00/AEI/10.13039/501100011033) and by the Comissionat per Universitats i Recerca de la Ge\-ne\-ralitat de Ca\-ta\-lu\-nya (grant No.\ 2021-SGR-01015). RA and FP were also partially supported by the Spanish State Research Agency through the Severo Ochoa and María de Maeztu Program for Centers and Units of Excellence in R\&D (CEX2020-001084-M).
HT and EV were partially supported by the Knut and Alice Wallenberg Foundation (KAW 2021.0140). EV was also supported by the Fields Institute for Research in Mathematical Sciences.
}

\subjclass[2010]%
{Primary
46L05; 
Secondary
19K14, 
46L80, 
46L85. 
}
\date{\today}

\begin{document}

\begin{abstract}
We demonstrate that pure \ca{s} form a robust class by proving that pureness follows from very weak comparison and divisibility properties.
Using this, we show that every simple, non-elementary \ca{} with a unique quasitrace and with very mild comparison is pure, and, as a result, has strict comparison.
Furthermore, sufficiently non-commutative \ca{s} of stable rank one and with weak comparison are likewise pure.

We also show that adequately non-elementary \ca{s} with finite nuclear dimension are pure, which leads to the verification of the non-simple Toms-Winter conjecture for a large class of \ca{s}.
\end{abstract}

\dedicatory{Dedicated to Mikael R{\o}rdam on the occasion of his $\mathit{65^{\mathit{th}}}$ birthday}

\maketitle

\section{Introduction}

Pureness is a regularity property for \ca{s} introduced by Winter in his seminal investigation  into $\mathcal{Z}$-stability and finite nuclear dimension of simple, nuclear \ca{s} \cite{Win12NuclDimZstable}.
Interest in these concepts was sparked by Toms' groundbreaking examples \cite{Tom08ClassificationNuclear}, where he constructed two simple, nuclear \ca{s} with the same $K$-theoretic and tracial data ---one pure, $\mathcal{Z}$-stable (i.e., tensorially absorbing the Jiang-Su algebra $\mathcal{Z}$), and of finite nuclear dimension, and the other not.~This led to a revision of the Elliott program, aiming at classifying separable, simple, nuclear \ca{s}.
Building on decades of research by numerous mathematicians, this was ultimately achieved in a series of remarkable results \cite{Win14LocalEllConj, TikWhiWin17QDNuclear, GonLinNiu20ClassifZstable2}.
(See also the recent approach \cite{CarGabSchTikWhi23arX:ClassifHom1} and the survey \cite{Whi23AbstractClassif}.)
The classification theorem states that  separable, simple, nuclear \ca{s} that are $\mathcal{Z}$-stable and satisfy the universal coefficient theorem (UCT) can be distinguished up to isomorphism by their Elliott invariant, which essentially consists of topological K-Theory and the tracial simplex.

As defined by Winter, a \ca{} is \emph{pure} if it has \emph{strict comparison} and is \emph{almost divisible}.
These are \ca{ic} analogs of two fundamental properties of projections in $\mathrm{II}_1$ factors discovered by Murray and von Neumann:
first, the comparison of projections in a $\mathrm{II}_1$ factor is determined by its unique trace $\tau$.
Second, every projection $p$ is divisible in the sense that for every $n \geq 2$ there exists a projection $q$ with $\tau(q)=\tfrac{1}{n}\tau(p)$.
Strict comparison for \ca{s} was first considered by Blackadar \cite{Bla88Comparison}, and is currently a well-established concept in the structure theory of \ca{s}.~This property is difficult to prove and its validty has far-reaching implications (see, for example, its role in the Toms-Winter conjecture below~\ref{subsec:TW}).~In fact, while comparison and divisibility hold automatically in every $\mathrm{II}_1$ factor, both strict comparison and almost divisibility can fail in simple \ca{s} \cite{Vil98SimpleCaPerforation, Vil99SRSimpleCa, DadHirTomWin09ZNotEmb}.

With these parallels in mind, one can think of finite, simple, pure \ca{s} as \ca{ic} analogs of $\mathrm{II}_1$ factors.
The concept also makes sense for infinite (that is, not finite) \ca{s}, and it turns out that a simple \ca{} is purely infinite if and only if it is infinite and pure, and one can think of such algebras as the \ca{ic} analogs of $\mathrm{III}$ factors.
More generally, nonsimple pure \ca{s} may be considered as analogs of von Neumann algebras with trivial type $\mathrm{I}$ summand.

Nowadays, the importance of pureness goes beyond its original intent. 
For separable \ca{s}, pureness of the central sequence algebra \cite{KirRor14CentralSeq} is equivalent to $\mathcal{Z}$-stability of the underlying \ca{} (\cite{PerThiVil24pre:CentrallyPure}); 
for non-nuclear \ca{s} ---such as those coming from non-amenable groups or certain infinite reduced free products--- pureness becomes an instrument weaker than $\mathcal{Z}$-stability for establishing strict comparison; 
for a \ca{} with mild divisibility (resp. comparison) properties, showing that it is \emph{not} pure implies that is has very poor comparison (resp. divisibility) in the precise sense of \cref{ThmIntro:CharPure} below. 
Due to these and other applications, the study of pureness has produced many important open questions: 
Is there a non-simple version of the Toms-Winter conjecture (\ref{subsec:NonSimpTW})? 
Do all $\mathrm{C}^*$-simple groups have strict comparison? 
(Equivalently, are they all pure? It is known that many are, as shown in \cite{AmrGaoKunPat24}, but in general the question is open.)
Are there real rank zero versions of R{\o}rdam's example \cite{Ror03FinInfProj}? 
(Conversely, are all non-elementary, simple, real rank zero \ca{s} pure?)

It is thus of utmost importance to understand which \ca{s} are pure. 
In this paper, we provide a new approach to the study of pureness by analyzing two new (and considerably weak) conceptual properties which we term \emph{controlled comparison} and \emph{functional divisibility}. 
The study of these notions ---discussed in \ref{subsec:CCFD} below--- allows us, among other things, to prove robustness results for the class of pure \ca{s}. 
One of the main applications of this new approach is the uncovering of new families where pureness can be verified. (The notions of nowhere scatteredness and the Global Glimm property mentioned below can be thought of as suitable
generalizations of non-elementariness to the non-simple case; see the discussion before; see the discussion before \cref{qst:TW}.)

\begin{thmIntro}[{\ref{prp:PureSR1}, \ref{prp:PureMonotrace}}]
\label{ThmB}
Let $A$ be a \ca{} with controlled comparison. 
Moreover, assume that
\begin{enumerate}[{\rm (i)}]
\item 
$A$ is nowhere scattered with real rank zero or stable rank one, or
\item 
$A$ is simple, unital, non-elementary, with a unique quasitracial state.
\end{enumerate}
Then $A$ is pure. In particular, $A$ has strict comparison.
\end{thmIntro}

\begin{thmIntro}[{\ref{prp:GGP-finNucDim-pure}}]
\label{thmD}
Every \ca{} with the Global Glimm Property and finite nuclear dimension is pure.
\end{thmIntro}

\cref{thmD} naturally fits in the context of the non-simple Toms-Winter conjecture (\ref{subsec:NonSimpTW}), while \cref{ThmB} can be applied to large families of \ca{s}: 
The conditions in (ii) are satisfied by any exact $\mathrm{C}^*$-simple group (\cite{BreKalKenOza17CSimpleUniqueTr}) and by Villadsen algebras of the second type \cite{Vil99SRSimpleCa}, while (i) is known to hold for all Villadsen algebras of the first type \cite{Vil98SimpleCaPerforation}.
Thus, exact $\mathrm{C}^*$-simple groups and Villadsen algebras have strict comparison if and only if they have controlled comparison, if and only if they are pure.

As previously stated, it is an important open problem to determine if all non-elementary simple real rank zero \ca{s} have strict comparison (or if they are $\mathcal{Z}$-stable under the assumption of nuclearity and separability). Part (i) of \cref{ThmB} also applies to this class of algebras, reducing the problem to proving controlled comparison.

\subsection{Controlled comparison and functional divisibility}\label{subsec:CCFD}

Together with his notion of pureness, Winter introduced in \cite{Win12NuclDimZstable} the quantified versions of \emph{$(m,n)$-pureness}, where $m$ and $n$ are non-negative integers measuring how big the failure of comparison ($m$) and divisibility ($n$) is. 
Surprisingly, it follows from Winter's work that for simple \ca{s} with locally finite nuclear dimension there is a dimension reduction phenomenon: 
$(m,n)$-pureness automatically implies $(0,0)$-pureness, that is, pureness. 
This was recently extended to the simple (possibly non-nuclear) setting in \cite[Theorem~10.5]{AntPerRobThi24TracesUltra}.
This dimension reduction phenomenon is one of the main techniques for establishing pureness or strict comparison outside the $\mathcal{Z}$-stable case, and a consequence of our \cref{ThmIntro:CharPure} is that it holds for all \ca{s}.


The novelty of this paper lies in the study of two unquantified notions which we call \emph{controlled comparison} and \emph{functional divisibility}. Each of these notions is weaker than any of their quantified counterparts, thus making them easier to check in concrete examples.

Inspired by the study initiated in \cite{RobTik17NucDimNonSimple}, we show that there is a beautiful interplay between these functional notions and relative central sequences. This ultimately allows us to bypass any ideal-lattice obstruction and obtain our main tool for verifying pureness:

\begin{thmIntro}[{\ref{prp:PureMain}}]
\label{ThmIntro:CharPure}
Given a \ca{} $A$, the following statements are equivalent:
\begin{enumerate}[{\rm (i)}]
\item
$A$ is pure, that is, it has strict comparison and is almost divisible.
\item
$A$ has controlled comparison and is functionally divisible. 
\end{enumerate}
\end{thmIntro}

\cref{ThmIntro:CharPure} implies, in particular, that the dimension reduction phenomenon alluded to above holds in full generality: 
If a \ca{} is $(m,n)$-pure for some $m$ and $n$, then it is pure. 

The present result provides a general statement that can be applied to non-simple, non-nuclear, \ca{s} such as, for example, central sequence algebras. 
Even in the simple case, \cref{ThmIntro:CharPure} represents a substantial improvement of \cite[Theorem~10.5]{AntPerRobThi24TracesUltra}: 
Given a \ca{} $A$ satisfying condition~(i) or~(ii) in \cref{ThmB}, it is not known if $A$ satisfies the divisibility property in $(m,n)$-pureness (known as \emph{$n$-almost divisibility}) for any $n$. 
However, we show in \cref{prp:FD-SR1,prp:FD-Monotrace} that any \ca{} satisfying condition~(i) or~(ii) in \cref{ThmB} is automatically functionally divisibile.

\subsection{The Toms-Winter conjecture}\label{subsec:TW}

This conjecture \cite{Win18ICM} aims at establishing a class of \ca{s} covered by the classification theorem mentioned at the beginning of the introduction.
Explicitly, it predicts that for a separable, simple, unital, non-elementary, nuclear \ca{}, the following properties are equivalent:
\begin{enumerate}
\item
The \ca{} has finite nuclear dimension, a noncommutative generalization of finite covering dimension introduced in \cite{WinZac10NuclDim}.
\item 
The \ca{} is $\mathcal{Z}$-stable.
\item
The \ca{} has strict comparison of positive elements (see \cref{sec:comparison}).
\end{enumerate}

In two groundbreaking papers \cite{Win12NuclDimZstable, CasEviTikWhiWin21NucDimSimple}, it was shown that conditions~(1) and~(2) are equivalent.
Further, it is known that~(2) implies~(3) \cite{Ror04StableRealRankZ}.
The implication (3)$\Rightarrow$(2) has been verified under certain additional assumptions \cite{Sat12arx:TraceSpace, KirRor14CentralSeq, TomWhiWin15ZStableFdBauer, Thi20RksOps}, but remains open in general.

The Toms-Winter conjecture is closely related to pureness, which we view as a property lying between~(2) and~(3).
Indeed, $\mathcal{Z}$-stability implies pureness, which in turn implies strict comparison of positive elements; 
see \cref{prp:ZStablePure}.
Further, strict comparison of positive elements implies pureness under the additional assumption of a positive solution to the Rank Problem; 
see \cref{sec:prelimns}.
Such a positive solution has been obtained for \ca{s} of real rank zero \cite{EllRor06Perturb} and of stable rank one \cite{Thi20RksOps, AntPerRobThi22CuntzSR1}.

\subsection{The non-simple Toms-Winter conjecture}\label{subsec:NonSimpTW}

The regularity properties of the Toms-Winter conjecture are also relevant for non-simple \ca{s}, and an investigation of their interplay was initiated by Robert and Tikuisis \cite{RobTik17NucDimNonSimple}.
While the assumption of unitality in the original formulation of the Toms-Winter conjecture could be dispensed with, the condition of being non-elementary is crucial.
Indeed, elementary \ca{s} have nuclear dimension zero, but they are neither $\mathcal{Z}$-stable nor pure.

The natural generalization of non-elementariness to the non-simple setting is the notion of nowhere scateredness: a \ca{} is \emph{nowhere scattered} if none of its quotients has a nonzero, elementary ideal, or equivalently if every hereditary sub-\ca{} is generated by nilpotent elements \cite{ThiVil24NowhereScattered}. 
This notion is tightly related to the Global Glimm Property: 
a \ca{} has the \emph{Global Glimm Property} if every hereditary sub-\ca{} contains an approximately full, nilpotent element \cite{ThiVil23Glimm}.~Clearly, every \ca{} with the Global Glimm Property is nowhere scattered. The converse remains open, and is known as the \emph{Global Glimm Problem}. In light of these considerations, the following question is natural:

\begin{qstIntro}
\label{qst:TW}
Let $A$ be a separable, nowhere scattered, nuclear \ca{}.
Are the following properties equivalent?
\begin{enumerate}
\item[(1)]
$A$ has finite nuclear dimension.
\item[(2)]
$A$ is $\mathcal{Z}$-stable.
\item[(3a)]
$A$ is pure.
\item[(3b)]
$A$ has strict comparison of positive elements.
\end{enumerate}
\end{qstIntro}

The implications (2)$\Rightarrow$(3a)$\Rightarrow$(3b) hold in general, and pureness always implies the Global Glimm Property;
see \cref{prp:ZStablePure}. Therefore, \cref{qst:TW} contains in particular the Global Glimm Problem for \ca{s} with finite nuclear dimension.

Partial results on the implication (3a)$\Rightarrow$(2) were established in \cite[Theorems~7.10, 7.15]{RobTik17NucDimNonSimple}, while the implication (2)$\Rightarrow$(1) has been shown to hold in a variety of cases; see, for example, \cite{BoeLi24NucDimSubhomTwGpd,BosGabSimWhi22NucDimOStableMaps,EllNiuSanTik20drASH}. As for simple \ca{s}, whether~(3b) implies~(3a) is related to the Rank Problem.

Combining \cref{thmD} with some of the above mentioned results, we obtain an answer to \cref{qst:TW} for a relevant class of \ca{s}, thereby verifying the non-simple Toms-Winter conjecture in this setting.

\begin{thmIntro}[{\ref{prp:NonsimpleTW}}]
\label{ThmE}
Let $A$ be a separable, locally subhomogeneous \ca{} that has stable rank one, topological dimension zero, and the Global Glimm Property.
Then, the answer to \cref{qst:TW} is positive.
\end{thmIntro}

\subsection*{Conventions}

Given a \ca{}, we let $A_+$ to denote its positive elements.
We let $\NN$ denote the set of natural numbers, including $0$.
Further, $\Cpct$ denotes the \ca{} of compact operators on a separable, infinite-dimensional Hilbert space.

\subsection*{Acknowledgements}

The authors thank Leonel Robert, Stuart White and Wilhelm Winter for valuable feedback on  earlier versions of this paper.

\section{Preliminaries}
\label{sec:prelimns}

The techniques used in this paper mainly concern the structure of the Cuntz semigroup of a \ca{} whose definition and properties we describe below.

\medskip

Given positive elements $a,b$ in a \ca{} $A$, one says that $a$ is \emph{Cuntz subequivalent} to $b$, denoted by $a \precsim b$, if there is a sequence $(r_n)_n$ in $A$ such that $\lim_{n\to\infty}\|a-r_nbr_n^*\|=0$. Further, $a$ and $b$ are \emph{Cuntz equivalent}, denoted $a \sim b$, if $a \precsim b$ and $b \precsim a$.
These relations were introduced by Cuntz in \cite{Cun78DimFct}. 

The \emph{Cuntz semigroup} $\Cu(A)$ of $A$ is defined as $\Cu(A)=(A\otimes \mathcal{K})_+/{\sim}$, equipped with the partial order induced by $\precsim$, and addition induced by addition of orthogonal positive elements. 

It is proved in \cite{CowEllIva08CuInv} that $\Cu(A)$ satisfies the following properties (usually referred to as axioms):

\begin{enumerate}
\item[\axiomO{1}]
If $(x_n)_n$ is an increasing sequence in $\Cu(A)$, then $\sup_n x_n$ exists.
\item[\axiomO{2}]
For any $x\in \Cu(A)$ there exists a sequence $(x_n)_n$ such that $x_n\ll x_{n+1}$ for all $n$ and $x=\sup_n x_n$.
(We say that $(x_n)_n$ is a $\ll$-increasing sequence.)
\item[\axiomO{3}]
If $x_1\ll x_2$ and $y_1\ll y_2$ in $\Cu(A)$, then $x_1+y_1\ll x_2+y_2$.
\item[\axiomO{4}]
If $(x_n)_n$ and $(y_n)_n$ are increasing sequences in $\Cu(A)$, then $\sup_n (x_n+y_n) = \sup_n x_n + \sup_n y_n$.
\end{enumerate}

The relation $\ll$ in these axioms is defined as follows:
$x\ll y$ if for every increasing sequence $(y_n)_n$ satisfying $y\leq \sup_n y_n$ there exists $n_0\in\NN$ such that $x\leq y_{n_0}$.
The relation $\ll$ is called the \emph{way-below relation}, or the \emph{compact containment relation}, and one says that `$x$ is way-below $y$' (or that `$x$ is compactly contained in $y$') if $x \ll y$.
An element $u\in S$ such that $u\ll u$ is termed \emph{compact}.

A positively ordered monoid satisfying axioms \axiomO{1}-\axiomO{4} is called a \emph{\CuSgp{}} and belongs to a category of semigroups that has thoroughly been studied (see, for example, \cite{AntPerThi18TensorProdCu, AntPerThi20CuntzUltraproducts, AntPerThi20AbsBivariantCu, AntPerThi20AbsBivarII}, among others; see also the recent survey \cite{GarPer23arX:ModernCu}). 
Yet, when dealing with semigroups obtained as Cuntz semigroups of \ca{s}, additional axioms and properties are satisfied:

\begin{enumerate}
\item[\axiomO{5}]
For all $x',x,y$ with $x'\ll x\leq y$ there exists $z$ such that $x'+z\leq y\leq x+z$. 
Moreover, if $x+w\leq y$ for some $w$, and $w'\ll w$, then $z$ may be chosen such that $w'\ll z$. 
\end{enumerate}

This axiom is usually referred to as the axiom of \textit{almost algebraic order}. For the Cuntz semigroup of an arbitrary \ca{}, the first part of the statement above was proved in \cite[Lemma~7.1]{RorWin10ZRevisited}, and the second part can be found in \cite[Proposition~4.6]{AntPerThi18TensorProdCu}.

\medskip 

Since we will often use the following consequence of \axiomO{5}, we add a proof here for completeness.

\begin{lma}[{\cite[Lemma~2.2]{ThiVil23Glimm}}]
\label{prp:DivO5}
Let $S$ be a \CuSgp{} satisfying \axiomO{5}, let $k\in\NN$, and let $z',z,x \in S$ satisfy $z' \ll z$ and $(k+1)z \leq x$.
Then there exists $d \in S$ such that 
\[
kz' + d \leq x \leq kz + d, \andSep 
x \leq (k+1)d.
\]
\end{lma}
\begin{proof}
Choose $z'' \in S$ such that $z' \ll z'' \ll z$. 
Applying \axiomO{5} for
\[
kz' \ll kz'' ,\quad 
z'' \ll z, \andSep
kz'' + z \leq x,
\]
we obtain $d \in S$ satisfying 
\[
kz' + d \leq x \leq kz'' + d, \andSep
z'' \ll d.
\]
Then $x \leq (k+1)d$, which shows that $d$ satisfies the desired conditions.
\end{proof}

Two more axioms are used very often in the theory of Cu-semigroups.
The first one, \axiomO{6}, is a weaker form of \textit{Riesz decomposition} and the second one, \axiomO{7}, is a weak form of interpolation. 
That the Cuntz semigroup of any \ca{} satisfies both \axiomO{6} and \axiomO{7} was proved in \cite[Proposition~5.1.1]{Rob13Cone} and \cite[Proposition~2.2]{AntPerRobThi21Edwards}, respectively. 
We will not need the precise formulation of these axioms here, but rather a very useful consequence of them that allows to take infima with idempotent elements. 
We make this precise below.

Recall that a \CuSgp\ $S$ is said to be \textit{countably based} if all elements can be written as suprema of elements from a fixed countable subset.
It is known that $\Cu(A)$ is countably based whenever $A$ is separable;
see \cite[Lemma~1.3]{AntPerSan11PullbacksCu}.

\begin{thm}[{\cite[Theorems~2.4, 2.5]{AntPerRobThi21Edwards}}]
\label{thm:O7}
Let $S$ be a countably based \CuSgp{} satisfying \axiomO{5}-\axiomO{7}.
Then each $x\in S$ and each idempotent element $w\in S$ have an infimum $x\wedge w$ in $S$. Further, the map $S\to S$ given by $x\mapsto x\wedge w$ is a monoid homomorphism preserving the order and the suprema of increasing sequences.
\end{thm}

Let $S$ be a \CuSgp{}. 
A map $\lambda\colon S\to [0,\infty]$ is called a \emph{functional} if $\lambda$ preserves addition, the zero element, order and suprema of increasing sequences.
We denote by $F(S)$ the set of functionals on $S$.
This is a cone when endowed with the operations of pointwise addition and pointwise scalar multiplication by positive real numbers. 
When equipped with the appropriate topology, $F(S)$ becomes a compact Hausdorff cone (see \cite[Theorems~4.4, 4.8]{EllRobSan11Cone}, \cite[Section~2.2]{Rob13Cone};
see also \cite[Theorem~3.17]{Kei17CuSgpDomainThy}).

Given $x\in S$, we obtain a map $\widehat{x}\colon F(S)\to [0,\infty]$ defined by $\widehat{x}(\lambda)=\lambda(x)$, for $\lambda\in F(S)$. 
The map $\widehat x$ is linear and lower semicontinuous. 
Denote by $\Lsc(F(S))$ the set of all linear, lower semicontinuous maps on $F(S)$ with values in $[0,\infty]$. 

Two main questions arise in this setting. 
The first concerns the extent to which the ordering in $A$ can be inferred from the ordering in $\textrm{Lsc}(F(S))$. 
Note that whilst $x\mapsto \widehat{x}$ is an additive map that preserves the order and suprema of increasing sequences, it typically does not preserve the way-below relation and, moreover, is generally not an order embedding. 
Still, the result below will be useful.

Given elements $x$ and $y$ in a partially ordered semigroup, one writes $x <_s y$ if there exists $k \in \NN$ such that $(k+1)x \leq ky$. 

\begin{prp}[{\cite[Lemma~2.2.5]{Rob13Cone}, \cite[Lemma~5.3]{AntPerRobThi24TracesUltra}}]
\label{prp:wayBelowLFS}
Let $S$ be a \CuSgp{} satisfying \axiomO{5}, and let $x,y \in S$. Then:
\begin{enumerate}[{\rm(i)}]
\item 
If $x \ll y$, and $s,t\in(0,\infty]$ satisfy $s<t$, we have $s\widehat{x} \ll t\widehat{y}$ in $\Lsc(F(S))$.
\item 
If $\widehat{x}\leq\gamma\widehat{y}$ for some $\gamma\in (0,1)$, then $x'<_s y$ for all $x' \in S$ with $x'\ll x$.
\end{enumerate}
\end{prp}

The second question that arises consists of determining the range of the map $x\mapsto \widehat{x}$. 
This is referred to as the \textit{Rank Problem} and was explored in detail in \cite{Thi20RksOps} and \cite{AntPerRobThi22CuntzSR1}.
For a \CuSgp{} $S$, let us denote by $\textrm{L}(F(S))$ the smallest subsemigroup of $\Lsc(F(S))$ closed under suprema of increasing sequences and containing all elements of the form $\tfrac{1}{n}\widehat{x}$ (see \cite[Section 3]{Rob13Cone}). 
Recall that an element $x \in S$ is termed \textit{soft} if for every $x' \in S$ with $x' \ll x$ there exists $t \in S$ such that $x' + t \ll x$ and $x' \ll \infty t$
(see \cite[Definition~4.3, Remark~4.4]{ThiVil24SoftOps}; 
for \ca{s} of stable rank one, this agrees with \cite[Definition~5.3.1]{AntPerThi18TensorProdCu}).

The following theorem summarizes the realization results from \cite{AntPerRobThi22CuntzSR1}, which will be required in \cref{sec:divisibility}.
The existence of the map $\alpha$ is shown in \cite[Theorem~7.2]{AntPerRobThi22CuntzSR1}, and the properties of $\alpha$ are proved in \cite[Proposition~7.4, Theorem~7.13, Proposition~8.2, Theorem~8.4, Corollary~8.5]{AntPerRobThi22CuntzSR1}.

\begin{thm}[{\cite{AntPerRobThi22CuntzSR1}}]
\label{thm:realization-ranks} 
Let~$A$ be a separable, nowhere scattered \ca{} with stable rank one. 
Then there is a natural, additive, order-preserving map $\alpha\colon \textrm{L}(F(\Cu(A))) \to \Cu(A)$ that preserves suprema of increasing sequences and such that $f=\widehat{\alpha(f)}$ for every $f \in \textrm{L}(F(\Cu(A)))$.

The image of $\alpha$ consists of soft elements, and we have 
\[
\alpha(f+\widehat x)=\alpha(f)+x
\] 
whenever $x \in \Cu(A)$ and $f \in \textrm{L}(F(\Cu(A)))$ satisfy $\widehat x\leq \infty f$. 
\end{thm}

\section{Variants of comparison}
\label{sec:comparison}

In this section we deal with various comparison properties. 
In particular, we recall the definition of \emph{$m$-comparison} (first considered in \cite[Lemma~6.1]{TomWin09Villadsen} and later formalized in \cite[Definition~2.8]{OrtPerRor12CoronaStability}) and clarify the equivalence between  $0$-comparison (also known as almost unperforation) with the notion of strict comparison for \ca{s}. 
We also recall the concept introduced in \cite[Definition~6.1]{AntPerRobThi24TracesUltra}, which in this paper is renamed as \emph{controlled comparison}.

\begin{dfn}
\label{dfn:m-comparison}
Let $m \in \NN$.~A \CuSgp{} $S$ has \emph{$m$-comparison} if for all $x,y_0,\ldots,y_m \in S$ satisfying $x <_s y_j$ for $j=0,\ldots,m$, we have $x \leq y_0 + \ldots + y_m$. 

We say that a \ca{} has \emph{$m$-comparison} if its Cuntz semigroup does.
A \CuSgp{} or \ca{} is said to be \emph{almost unperforated} if it has $0$-comparison.
\end{dfn}

\begin{rmks}
\label{rmk:m-comparison}
(1) 
Robert showed in \cite[Theorem~1.3]{Rob11NuclDimComp} that a \ca{} has $m$-comparison whenever it has nuclear dimension at most $m$.

(2) 
Let $m\in\NN$ and let $S$ be a \CuSgp{} that has $m$-comparison.
Then, $x<_sy$ implies $x \leq (m+1)y$, for all $x,y \in S$.
\end{rmks}

The notion of strict comparison was first considered in \cite{Bla88Comparison}.~For \CuSgp{s}, the concept was considered in \cite[Section~6]{EllRobSan11Cone}.

\begin{dfn}
\label{dfn:str-comparison}
A \CuSgp{} $S$ has \emph{strict comparison} if for all $x,y\in S$ satisfying $x\leq \infty y$ and $\lambda(x)<\lambda(y)$ for any $\lambda\in F(S)$ with $\lambda(y)=1$, one has $x \leq y$.

A \ca{} $A$ has \emph{strict comparison} of positive elements by quasitraces (specifically, $[0,\infty]$-valued, lower semicontinuous $2$-quasitraces) if its Cuntz semigroup does. 
\end{dfn}

\begin{pgr}[Strict comparison and almost unperforation]
\label{rmk:str-comparison}
We take the opportunity here to clarify the equivalence between the concepts of strict comparison and almost unperforation. 
Let $S$ be a \CuSgp{}, and let $x,y\in S$. The following conditions relating the comparison of $x,y$ and $\widehat{x}, \widehat{y}$  were considered in \cite{AntPerThi18TensorProdCu}:
\begin{enumerate}
\item[(i)]
We have $x <_s y$, that is, there exists $k \in \NN$ such that $(k+1)x \leq ky$.
\item[(ii)]
We have $\widehat{x} <_s \widehat{y}$, that is, there exists $\varepsilon \in (0,1)$ such that $\widehat{x} \leq (1-\varepsilon)\widehat{y}$.
\item[(iii)]
We have $x \leq \infty y$ and $\lambda(x)<\lambda(y)$ for any $\lambda\in F(S)$ with $\lambda(y)=1$.
\item[(iv)]
We have $x' <_s y$ for every $x' \in S$ with $x' \ll x$.
\end{enumerate}
It was shown in \cite[Theorem~5.2.18]{AntPerThi18TensorProdCu} that the implications (i)$\Rightarrow$(ii)$\Rightarrow$(iii)$\Rightarrow$(iv) always hold. 
Using this, it follows that the following conditions are equivalent for $S$.
(This was first obtained in \cite[Proposition~6.2]{EllRobSan11Cone}; see also \cite[Proposition~5.2.20]{AntPerThi18TensorProdCu}.)

\begin{enumerate}
\item[(i')]
$S$ is almost unperforated.
\item[(ii')]
If $x,y \in S$ satisfy $\widehat{x} \leq (1-\varepsilon)\widehat{y}$ for some $\varepsilon \in (0,1)$, then $x \leq y$.
\item[(iii')]
$S$ has strict comparison.
\end{enumerate}

\medskip

\noindent
Now, let $A$ be a \ca. There is a natural correspondence between functionals on $\Cu(A)$ and quasitraces on $A$. This is defined by sending a quasitrace $\tau \colon (A \otimes \mathcal{K})_+ \to [0,\infty]$ to the functional~$d_\tau \colon \Cu(A) \to [0,\infty]$ given by
\[
d_\tau([a]) = \lim_{n \to \infty} \tau(a^{1/n})
\]
for $a \in (A \otimes \mathcal{K})_+$;
see, for example, \cite[Proposition~4.2]{EllRobSan11Cone}.

Using this, it follows that $A$ has strict comparison if and only if for all $a,b \in (A \otimes \mathcal{K})_+$ such that $a$ belongs to the closed ideal generated by $b$, and such that $d_\tau([a]) < d_\tau([b])$ for every quasitrace $\tau$ with $d_\tau([b])=1$, we have $a \precsim b$.

It also follows from the above considerations that a \ca{} $A$ has strict comparison (of positive elements by quasitraces) if and only if it is almost unperforated.

Some of the confusion surrounding the notion of strict comparison arises since it is applied in a more restrictive setting, for example by considering comparison of projections (instead of positive elements), or by considering comparison by traces (instead of quasitraces);
see \cite[Section~3]{NgRob16CommutatorsPureCa} for some clarifications.

Another source of confusion is that strict comparison is often used in the context of simple \ca{s}, in which case every element is automatically contained in the closed ideal generated by a nonzero element.
Therefore, a simple, unital \ca{}~$A$ has strict comparison if and only if for all nonzero elements $a,b \in (A \otimes \mathcal{K})_+$ such that $d_\tau([a]) < d_\tau([b])$ for every normalized quasitrace $\tau$, we have $a \precsim b$.
\end{pgr}

For the definition below, we need to recall the notion of a scale in a \CuSgp.

\begin{pgr}[Scales]
Let $S$ be a \CuSgp.
Following \cite[Definition~4.1]{AntPerThi20CuntzUltraproducts}, a \emph{scale} in $S$ is a subset $\Sigma \subseteq S$ satisfying the following conditions:
\begin{enumerate}[{\rm (i)}]
\item 
$\Sigma$ is downward hereditary: 
If $x \leq y$ for some $x \in S$ and $y \in \Sigma$, then $x \in \Sigma$.
\item 
$\Sigma$ is closed under suprema of increasing sequences.
\item 
$\Sigma$ generates $S$ as an ideal, that is, for $x',x \in S$ with $x' \ll x$, there exist $y_1,\ldots,y_n \in \Sigma$ such that $x' \leq y_1+\ldots+y_n$.
\end{enumerate}
We say that the pair $(S,\Sigma)$ is a \emph{scaled \CuSgp{}}.

Given a scale $\Sigma$ and $d \in \NN$, the \emph{$d$-fold amplification} of $\Sigma$ is defined as $\Sigma^{(0)} := \{0\}$, and for $d \geq 1$ as
\[
\Sigma^{(d)} 
:= \left\{ x \in S : \parbox{7cm}{for each $x' \in S$ with $x' \ll x$ there exist $y_1,\ldots,y_d \in \Sigma$ such that $x' \leq y_1+\ldots+y_d$} \right\}.
\]

If $A$ is a \ca{}, then 
\[
\Sigma_A := \big\{ x \in \Cu(A) : x \leq [a] \text{ for some $a \in A_+$} \big\}
\]
is a scale in $\Cu(A)$, and the pair $(\Cu(A),\Sigma_A)$ is called the \emph{scaled Cuntz semigroup} of $A$;
see \cite[Paragraph~4.2]{AntPerThi20CuntzUltraproducts} and \cite[Lemma~3.3]{ThiVil23Glimm}, and  also \cite[Paragraph~4.5]{AntPerRobThi24TracesUltra}.
By \cite[Proposition~7.4]{AntPerRobThi24TracesUltra}, for $d \geq 1$ we have
\[
\Sigma_A^{(d)} = \big\{ x \in \Cu(A) : x \leq [a] \text{ for some $a \in M_d(A)_+$} \big\}.
\]
\end{pgr}

Given $N\in\NN$ and elements $x$ and $y$ in a partially ordered semigroup, we write $x \leq_N y$ if $nx \leq ny$ for all $n\in\NN$ with $n \geq N$.
The following was termed `locally bounded comparison amplitude' in \cite[Definition~6.1]{AntPerRobThi24TracesUltra}.

\begin{dfn}
A scaled \CuSgp{} $(S,\Sigma)$ has \emph{controlled comparison} if for every $\gamma \in (0,1)$ and $d\in\NN$ there exists $N=N(\gamma,d)$ such that the following holds:
\[
\widehat{x}\leq\gamma\widehat{y} \quad \textrm{ implies } \quad x \leq_N y, \quad \text{ for all } x,y\in\Sigma^{(d)}.
\]

We say that a \ca{} $A$ has controlled comparison if $(\Cu(A),\Sigma_A)$ does.
\end{dfn}

The next result follows easily from \cite[Proposition~10.3]{AntPerRobThi24TracesUltra}.

\begin{prp}
\label{prp:m-comparison-CC}
If a \ca{} has $m$-comparison for some $m \in \NN$, then it has controlled comparison.
\end{prp}
\begin{proof}
In order to apply \cite[Proposition~10.3]{AntPerRobThi24TracesUltra} it is required that the Cuntz semigroup of a \ca{} satisfies \axiomO{5}, \axiomO{6}, and the so-called Edward's condition.
The first two are always satisfied, as mentioned in the preliminaries, and the third condition was shown to hold in \cite[Theorem~5.3]{AntPerRobThi21Edwards}.
\end{proof}	

\begin{cor}
\label{prp:FinNucDim-CC}
Every \ca{} with finite nuclear dimension has controlled comparison.
\end{cor}
\begin{proof}
Let $A$ be a \ca{} with nuclear dimension at most $m$ for some $m\in\NN$.
Then $A$ has $m$-comparison by \cite[Theorem~1.3]{Rob11NuclDimComp}, hence the result follows from \cref{prp:m-comparison-CC}.
\end{proof}

\section{Variants of divisibility}
\label{sec:divisibility}

In this section we introduce the concept of \emph{functional divisibility} and show it is a natural generalization of $n$-almost divisibility (as introduced in \cite[Definition~3.5]{Win12NuclDimZstable} for simple, unital \ca{s}, and later refined in \cite[Section~2.3]{RobTik17NucDimNonSimple} for general \ca{s}). We prove that functional divisibility is automatic for nowhere scattered \ca{s} of stable rank one, and also for simple, unital, non-elementary \ca{s} with a unique normalized quasitrace.

\begin{dfn}
\label{dfn:n-divisibility}
Let $n\in\NN$.
A \CuSgp{} $S$ is \emph{$n$-almost divisible} if for every $x' \ll x$ in $S$ and every $k \in \NN$ with $k \geq 1$, there exists $y \in S$ such that
\[
ky \leq x, \andSep
x' \leq (k+1)(n+1)y.
\]

We say that a \ca{} is \emph{$n$-almost divisible} if its Cuntz semigroup is.~A \CuSgp{} or \ca{} is said to be \emph{almost divisible} if it is $0$-almost divisible.~In other words, if for all $x' \ll x$ in $S$ and $k \in \NN$ with $k \geq 1$ there exists $y \in S$ such that $ky \leq x$, and $x' \leq (k+1)y$.
\end{dfn}

\begin{pgr}[Connections with earlier divisibility notions]
\label{rmk:n-divisibility}
In parts of the literature, `almost divisible' is defined as the slightly stronger condition that for every $x \in S$ and  $k \in \NN$ with $k \geq 1$ there exists $y \in S$ such that
\[
ky \leq x \leq (k+1)y.
\]

This version of `almost divisibility' implies the one from \cref{dfn:n-divisibility}, and for almost unperforated \CuSgp{s} both versions agree by \cite[Proposition~7.3.7]{AntPerThi18TensorProdCu}.
The advantage of the version from \cref{dfn:n-divisibility} is that it enjoys better permanence properties; 
see \cite[Remark~7.3.5]{AntPerThi18TensorProdCu}.

Almost divisibility is automatic in certain cases.~For example, it was shown in \cite[Theorem~9.1]{ThiVil24NowhereScattered} that every nowhere scattered \ca{}~$A$ of real rank zero is almost divisible. 
In fact, $\Cu(A)$ is even `weakly divisible' in the sense that for every $x\in \Cu(A)$, there are $y,z\in \Cu(A)$ such that $x=2y+3z$; 
see also \cite[Section~6]{AraGooPerSil10NonSimplePI}. 
In the simple case, and for the Murray-von Neumann semigroup, this was obtained in \cite[Proposition~5.3]{PerRor04AFembeddings}.
\end{pgr}

Inspired by the results in \cite[Section~6]{RobTik17NucDimNonSimple} and also in \cite{Win12NuclDimZstable}, we introduce the following more general divisibility condition in terms of comparison by functionals.

\begin{dfn}
\label{dfn:FuncDiv}
We say that a \CuSgp{} $S$ is \emph{functionally divisible} if the following conditions are satisfied:
\begin{itemize}
\item[(L)] 
For every $x' \ll x$ in $S$, every $k \in \NN$ with $k\geq 1$, and every $\varepsilon\in(0,1)$, there exists $y \in S$ such that
\[
ky \leq x, \andSep
(1-\varepsilon)\widehat{x'} \leq k\widehat{y}.
\]
\item[(U)] 
For every $x' \ll x$ in $S$, every $k \in \NN$ with $k \geq 1$, and every $\varepsilon\in(0,1)$, there exists $z \in S$ such that
\[
z \leq x,\quad 
(1-\varepsilon)k\widehat{z} \leq \widehat{x}, \andSep 
x' \leq kz.
\]
\end{itemize}
We say that a \ca{} $A$ is \emph{functionally divisible} if $\Cu(A)$ is.
\end{dfn}

\begin{rmk}
\label{rmk:FuncDiv}
Let $S$ be a \CuSgp{} satisfying \axiomO{5}, and let $x \in S$. 
While \cref{dfn:n-divisibility} captures notions of divisibility of $x$ in $S$, one may also consider the divisibility of $\widehat{x}$  in $\Lsc(F(S))$.
The straightforward concept would be to require that for every $x \in S$ and every $k \in \NN$ with $k \geq 1$, there exists $y \in S$ such that $k\widehat{y} = \widehat{x}$.
Taking also the way-below relation into account in order to obtain better permanence properties, one is led to the following notion: 
\begin{itemize}
\item[(D)] 
For every $x \in S$, every $f \in \Lsc(F(S))$ with $f \ll \widehat{x}$, and every $k \in \NN$ with $k \geq 1$, there exists $y \in S$ such that 
\[
f \leq k\widehat{y} \leq \widehat{x}.
\]
\end{itemize}

Using \cref{prp:wayBelowLFS}(i), we see that a function $f \in \Lsc(F(S))$ satisfies $f \ll \widehat{x}$ if and only if there exists $x' \in S$ with $x' \ll x$ and $\varepsilon>0$ such that $f \leq (1-\varepsilon)\widehat{x'}$.
Thus, condition~(D) holds if and only if for every $x' \ll x$ in $S$, every $k \in \NN$ with $k \geq 1$, and every $\varepsilon>0$, there exists $y \in S$ such that 
\[
(1-\varepsilon)\widehat{x'} \leq k\widehat{y} \leq \widehat{x}.
\]

Conditions~(L) and~(U) in \cref{dfn:FuncDiv} are natural strengthenings of this, where one of the inequalities on functionals is upgraded to an inequality in $S$.
The `L' stands for `lower', since in this condition the element $ky$ is smaller than $x$, while the `U' stands for `upper' since here $ky$ is larger than $x'$.
\end{rmk}

The next result shows that it suffices to consider $k=2$ in order to verify condition~(L) in \cref{dfn:FuncDiv}.

\begin{prp}
\label{prp:CharLFD}
Let $S$ be a \CuSgp{} satisfying \axiomO{5}.
Then the following statements are equivalent:
\begin{enumerate}[{\rm (i)}]
\item
$S$ satisfies condition~(L) in \cref{dfn:FuncDiv}.
\item
For every $x' \ll x$ in $S$ and $k \in \NN$, there exists $y \in S$ such that
\[
ky \leq x, \andSep
\widehat{x'} \leq \left(k+1\right)\widehat{y}.
\]
\item
For every $x' \ll x$ in $S$ and every $\varepsilon \in (0,1)$ there exists $y \in S$ such that
\[
2y \leq x, \andSep
(1-\varepsilon)\widehat{x'} \leq 2\widehat{y}.
\]
\end{enumerate}
\end{prp}
\begin{proof}
Assuming~(i), let us verify~(ii).
For $k=0$, we use $y=x$.
For $k \geq 1$, we apply~(i) with $\varepsilon>0$ small enough such that $(1-\varepsilon)^{-1}k \leq k+1$.
Similarly, to see that~(ii) implies~(iii), for a given $\varepsilon \in (0,1)$ apply~(ii) with $2k$ where $k$ is large enough such that $1-\varepsilon \leq (1+\tfrac{1}{2k})^{-1}$.

\medskip

Next, assuming that~(iii) holds, let us verify~(i).
First, by induction over $m$, we verify that the following statement holds:
\begin{enumerate}
\item[(L$_{2^m}$)]
For every $x' \ll x$ in $S$ and $\varepsilon \in (0,1)$ there exists $y \in S$ such that
\[
2^my \leq x, \andSep
(1-\varepsilon)\widehat{x'} \leq 2^m\widehat{y}.
\]
\end{enumerate}
The case $m=1$ holds by assumption.
Assume that (L$_{2^m}$) holds for some $m \geq 1$.
To verify (L$_{2^{m+1}}$), let $x' \ll x$ in $S$ and $\varepsilon \in (0,1)$.
Pick $x'' \in S$ and $\delta>0$ such that
\[
x' \ll x'' \ll x, \andSep
(1-\delta)^3 = 1-\varepsilon.
\]
Applying (L$_{2^m}$) for $x'' \ll x$ and $\delta$, we obtain $w \in S$ such that
\[
2^mw \leq x, \andSep
(1-\delta)\widehat{x''} \leq 2^m\widehat{w}.
\]
By \cref{prp:wayBelowLFS}(i), we have $(1-\delta)^2\widehat{x'} \ll (1-\delta)\widehat{x''}$ in $\Lsc(F(S))$, hence also $(1-\delta)^2\widehat{x'} \ll 2^m \widehat{w}$. This allows us to choose $w' \in S$ such that
\[
w' \ll w, \andSep
(1-\delta)^2\widehat{x'} \leq 2^m\widehat{w'}.
\]

Next, applying (iii) for $w' \ll w$ and $\delta$, we obtain $y \in S$ such that
\[
2y \leq w, \andSep
(1-\delta)\widehat{w'} \leq 2\widehat{y}.
\]
Then
\[
2^{m+1}y
\leq 2^mw
\leq x, \andSep
(1-\varepsilon)\widehat{x'}
= (1-\delta)^3\widehat{x'}
\leq (1-\delta)2^m\widehat{w'}
\leq 2^{m+1}\widehat{y},
\]
which verifies (L$_{2^{m+1}}$).

\medskip

Now, to prove~(i), let $x' \ll x$ in $S$, let $k \in \NN$ with $k \geq 1$, and let $\varepsilon \in (0,1)$.
Pick $x'' \in S$ and $\delta>0$ such that
\[
x' \ll x'' \ll x, \andSep
(1-\delta)^2 = 1-\varepsilon.
\]

Using that dyadic rationals are dense in $\mathbb{R}$, choose $n,m \geq 1$ such that
\[
(1-\delta)\frac{1}{k}
\leq \frac{n}{2^m}
\leq \frac{1}{k}.
\]
Then
\[
(1-\delta)2^m \leq kn \leq 2^m.
\]

Applying (L$_{2^{m}}$) for $x'' \ll x$ and $\delta$, we obtain $w \in S$ such that
\[
2^mw \leq x, \andSep
(1-\delta)\widehat{x'} \leq 2^m\widehat{w}.
\]
Set $y := nw$.
Then
\[
ky
= k(nw)
\leq 2^mw
\leq x, \andSep
(1-\varepsilon)\widehat{x'}
= (1-\delta)^2\widehat{x'}
\leq (1-\delta)2^m\widehat{w}
\leq kn\widehat{w}
= k\widehat{y},
\]
which shows that $y$ has the desired properties.
\end{proof}

\begin{rmk}
\label{rmk:osti}
Let $S$ be a \CuSgp{} satisfying \axiomO{5}.
It is not clear if there exists a characterization of condition~(U) in \cref{dfn:FuncDiv}  analogous to \cref{prp:CharLFD}.
The reason is that condition~(U) has the additional assumption that the dividing element is dominated by $x$, which is automatic in condition~(L).
	
On the other hand, for $m \geq 1$ let us consider the following condition:
\begin{itemize}
\item[($U_{2^m}$)] 
For every $x' \ll x$ in $S$ and every $\varepsilon\in(0,1)$, there exists $y \in S$ such that
\[
y \leq x,\quad 
(1-\varepsilon)2^m\widehat{y} \leq \widehat{x}, \andSep 
x' \leq 2^my.
\]
\end{itemize}
An argument as in the proof of \cref{prp:CharLFD} shows that if $S$ satisfies ($U_2$), then it satisfies ($U_{2^m}$) for all $m$. (See also \cref{rmk:adapt}.)
\end{rmk}

To show that $n$-almost divisible Cuntz semigroups are functionally divisible (\cref{prp:FD-NAlmDiv}), we first need some preparatory technical results.
We will use the concept of $(2,\omega)$-divisibility (see also \cref{sec:GGP}).~Given $k \in \NN$ with $k \geq 2$, a \CuSgp{} $S$ is said to be \emph{$(k,\omega)$-divisible} if for all $x' \ll x$ in $S$ there exists $y \in S$ such that
\[
ky \leq x, \andSep x' \leq \infty y.
\]

If $S$ is $(k,\omega)$-divisible for some $k \geq 2$, then it is $(2,\omega)$-divisible.
The converse also holds by \cite[Lemma~3.4]{ThiVil23Glimm}.
Thus, a \CuSgp{} is $(2,\omega)$-divisible if and only if it is $(k,\omega)$-divisible for all $k \geq 2$.

Condition (L') in the result below was considered in \cite[Theorem~6.1]{RobTik17NucDimNonSimple}.

\begin{lma}
\label{prp:CharLFD-2Omega}
Let $S$ be a \CuSgp{} satisfying \axiomO{5}-\axiomO{7}. 
Then $S$ satisfies condition~(L) in \cref{dfn:FuncDiv} if and only if $S$ is $(2,\omega)$-divisible and satisfies
\begin{enumerate}
\item[(L')] 
For every $x' \ll x$ in $S$, every $k \in \NN$ with $k \geq 1$, and every $\varepsilon>0$, there exists $y\in S$ such that
\[
ky \leq x, \andSep
\widehat{x'} \leq k\widehat{y}+\varepsilon\widehat{x}.
\]
\end{enumerate}	
\end{lma}
\begin{proof}
To show the forward implication, assume that $S$ satisfies
\begin{itemize}
\item[(L)] 
For every $x' \ll x$ in $S$, every $k \in \NN \setminus \{0\}$, and every $\varepsilon \in (0,1)$, there exists $y \in S$ such that
\[
ky \leq x, \andSep
(1-\varepsilon)\widehat{x'} \leq k\widehat{y}.
\]
\end{itemize}
To verify that~$S$ is $(2,\omega)$-divisible, let $x' \ll x$ in $S$.
Applying (L) with the given $x',x$, and also $k=2$ and $\varepsilon=\tfrac{1}{2}$, we obtain $y \in S$ such that 
\[
2y \leq x, \andSep
\frac{1}{2}\widehat{x'} \leq 2\widehat{y}.
\]
The latter condition implies that $\widehat{x'} \leq 4\widehat{y}$, and hence $x'\leq \infty y$ (using, for example, \cite[Lemma~6.6]{AntPerRobThi22CuntzSR1}). 
This shows that~$S$ is $(2,\omega)$-divisible.

To verify (L'), let $x' \ll x$ in $S$, let $k \in \NN$ with $k \geq 1$, and $\varepsilon>0$.
Pick $\delta \in (0,1)$ such that $(1-\delta)^{-1}=1+\varepsilon$.
Applying (L) with the given $x',x$ and $k$, and also $\delta$, we obtain $y \in S$ such that 
\[
ky \leq x, \andSep
(1-\delta)\widehat{x'} \leq k\widehat{y}.
\]
Using that $ky \leq x$, we get
\[
\widehat{x'} 
\leq (1-\delta)^{-1}k\widehat{y}
= (1+\varepsilon)k\widehat{y}
\leq k\widehat{y} + \varepsilon\widehat{x},
\]
as desired.

\medskip

To show the backward implication, assume that $S$ is $(2,\omega)$-divisible and satisfies~(L').
We verify condition~(ii) of \cref{prp:CharLFD}.
Let $x' \ll x$ in $S$, and let $k\in\NN$.
We need to find $y\in S$ such that
\[
ky \leq x, \andSep 
\widehat{x'} \leq (k+1)\widehat{y}.
\]
	
For $k=0$ and $k=1$ use $y := x$. Thus, from now on we may assume that $k \geq 2$. 
	
Using the methods from \cite[Section~5]{ThiVil21DimCu2}, we find a countably based, $(2,\omega)$-divisible sub-\CuSgp{} $H \subseteq S$ that satisfies \axiomO{5}-\axiomO{7} and contains $x'$, $x$.
	
Choose $x'' \in H$ and $n \in \NN$ with $n \geq 1$ such that
\[
x' \ll x'' \ll x, \quad
k+1 \leq n, \andSep
\frac{k(k+1)+n}{n} \leq 1+\frac{1}{k}.
\]
	
Since $H$ is $(2,\omega)$-divisible, it is also $(n,\omega)$-divisible by the arguments in \cite[Lemma~3.4]{ThiVil23Glimm}.
Applied to $x'' \ll x$ and $n$, we obtain $c \in H$ such that
\[
nc \leq x, \andSep
x'' \ll \infty c.
\]
Using that $x'' \ll \infty c$, we can choose $c' \in H$ such that
\[
x'' \ll \infty c', \andSep
c' \ll c.
\]
	
We have $(k+1)c \leq nc \leq x$. 
Using \cref{prp:DivO5}, we obtain $d \in H$ such that
\[
kc' + d \leq x \leq kc + d, \andSep
x \leq (k+1)d.
\]
	
Thus, one has
\[
nx''
\leq nx
\leq nkc + nd
\leq kx + nd
\leq (k(k+1)+n)d.
\]
	
Since $H$ is countably based and satisfies \axiomO{5}-\axiomO{7}, we may apply \cref{thm:O7}. 
Hence, for every $s \in H$ the infimum $s \wedge \infty c'$ exists (in $H$), and the map $s \mapsto s \wedge \infty c'$ is additive.
Set $e := d \wedge \infty c'$.
Since $x'' \ll \infty c'$, we have 
\begin{align*}
nx' 
\ll nx'' 
= (nx'') \wedge (\infty c')
&\leq (k(k+1)+n)d \wedge (\infty c') \\
&= \big( k(k+1)+n \big) \big( d \wedge (\infty c') \big)
= (k(k+1)+n)e.
\end{align*}
	
This allows us to find $e' \in H$ satisfying
\[
nx' \leq (k(k+1)+n)e', \andSep
e' \ll e.
\]
Let us further choose $e'' \in H$ such that $e'\ll e''\ll e$. 
	
Since $e\leq\infty c'$ and $e'\ll e$, we may pick $N\in\NN$ with $N \geq 1$ such that $e''\leq Nc'$. 
Then choose $\varepsilon>0$ small enough such that
\[
\left( 1+\frac{1}{k}\right) \varepsilon\leq \frac{1}{N}.
\]

Now, working again in $S$, and applying~(L') to $e' \ll e''$, $k$ and $\varepsilon$, we obtain $f \in S$ such that
\[
kf \leq e'', \andSep
\widehat{e'} \leq k\widehat{f} + \varepsilon\widehat{e''}.
\]
	
Set $y := c'+f$.
We have $kf \leq e \leq d$ and therefore
\[
ky
= kc' + kf
\leq kc' + d
\leq x.
\]
Further, we have
\begin{align*}
\widehat{x'}
&\leq \frac{k(k+1)+n}{n} \widehat{e'}
\leq  \left( 1+\frac{1}{k}\right)\widehat{e'} 
\leq \left( 1+\frac{1}{k}\right)\left( k\widehat{f} + \varepsilon\widehat{e''} \right) \\
&\leq (k+1)\widehat{f} + \frac{1}{N}\widehat{e''} 
\leq (k+1)\widehat{f} + \widehat{c'}
\leq (k+1)\left(\widehat{f} + \widehat{c'}\right)
= (k+1)\widehat{y},
\end{align*}
as desired.
\end{proof}

The next result shows that if a \emph{compact} element satisfies condition~(L') in \cref{prp:CharLFD-2Omega}, then it also satisfies condition~(U) in \cref{dfn:FuncDiv}.

\begin{lma}
\label{prp:CpctLFD}
Let $S$ be a \CuSgp{} satisfying \axiomO{5}, and let $x\in S$ be a compact element such that for every $k \in \NN$ with $k \geq 1$, and every $\varepsilon>0$, there exists $y \in S$ such that
\[
ky \leq x, \andSep
\widehat{x} \leq k\widehat{y}+\varepsilon\widehat{x}.
\]
	
Then, for every $k \in \NN$ with $k \geq 1$, and every $\delta>0$, there exists $z \in S$ such that
\[
z \leq x,\quad 
(1-\delta)k\widehat{z} \leq \widehat{x}, \andSep 
x \leq kz.
\]
\end{lma}
\begin{proof}
Given $k \in \NN$ with $k \geq 1$ and $\delta>0$, pick $\varepsilon>0$ such that
\[
1 + 2(k-1)\varepsilon \leq (1-\delta)^{-1}.
\]

Applying the assumption for $k$ and $\varepsilon$, we obtain $y \in S$ such that 
\[
ky \leq x, \andSep 
\widehat{x} \leq k\widehat{y}+\varepsilon\widehat{x}.
\]

Now, using that $x\ll x$, it follows from \cref{prp:wayBelowLFS}(i) that 
\[
(1-\varepsilon)\widehat{x}
\ll \widehat{x} 
\leq k\widehat{y}+\varepsilon\widehat{x}.
\]
This allows us to choose $y' \in S$ such that $y' \ll y$ and 
\[
(1-\varepsilon)\widehat{x} 
\leq k\widehat{y'} + \varepsilon\widehat{x}.
\]

Applying \cref{prp:DivO5} for $ky \leq x$ and $y' \ll y$, we get $z \in S$ such that 
\begin{equation*}
\label{eq2}
(k-1)y'+z \leq x \leq kz.
\end{equation*}
We now show that $z$ has the claimed properties.
It is clear that $z \leq x$ and $x \leq kz$.
It remains to verify that  $(1-\delta)k\lambda(z) \leq \lambda(x)$ for every $\lambda \in F(S)$.
This is clear if $\lambda(x)=\infty$.
Let $\lambda \in F(S)$ with $\lambda(x)<\infty$. 
From 
the fact that $(1-\varepsilon)\widehat{x} 
\leq k\widehat{y'} + \varepsilon\widehat{x}$, as shown above, we deduce that
\[
(1-2\varepsilon)\lambda(x) \leq k\lambda(y').
\]

Since $ky\leq x$, it follows that $\lambda (y')<\infty$ and, consequently, using our choice of $\varepsilon$ in the last step,
\begin{align*}
k\lambda(z) 
&\leq k\lambda(x) - k(k-1)\lambda(y')
\leq k\lambda(x) - (1-2\varepsilon)(k-1)\lambda(x) \\
&= (1+2(k-1)\varepsilon)\lambda(x)
\leq (1-\delta)^{-1}\lambda(x),
\end{align*}
as desired.
\end{proof}

\begin{prp}
\label{prp:FD-NAlmDiv}
If a \ca{} is $n$-almost divisible for some $n\in\NN$, then it is functionally divisible.
\end{prp}
\begin{proof}
Let $n \in \NN$, and let $A$ be a \ca{} that is $n$-almost divisible.
It is easy to deduce that $\Cu(A)$ is $(2,\omega)$-divisible.
Therefore, by \cref{prp:CharLFD-2Omega}, to show that $\Cu(A)$ is functionally divisible it suffices to verify condition (L') in \cref{prp:CharLFD-2Omega} and condition~(U) in \cref{dfn:FuncDiv}. 

Let $x' \ll x$ in $\Cu(A)$, let $k \in \NN$ with $k \geq 1$, and let $\varepsilon>0$.
We may assume that $\varepsilon=\tfrac{1}{N}$ for some integer $N\geq 1$, and we need to find $y \in \Cu(A)$ such that
\[
ky \leq x, \andSep
N\widehat{x'} \leq Nk\widehat{y} + \widehat{x},
\]
and $z \in \Cu(A)$ such that
\[
z \leq x, \quad
(N-1)k\widehat{z} \leq N\widehat{x}, \andSep
x' \leq kz.
\]
Without loss of generality, we may assume that $A$ is stable, which allows us to pick $a \in A_+$ with $x=[a]$.

Let $\mathcal{U}$ be a free ultrafilter on $\NN$. Denote by $A_{\mathcal{U}}$ the free ultrapower of $A$. Set $C=\{ a\}'\cap A_{\mathcal{U}}$ and $I=\{ a\}^{\perp}\cap A_{\mathcal{U}}$. 
Then, it follows from \cite[Theorem~6.1]{RobTik17NucDimNonSimple} that the class of the unit $[1]$ in $\Cu(C/I)$ satisfies condition~(L') in \cref{prp:CharLFD-2Omega}, that is, for every $l \in \NN$ with $l \geq 1$, and every $\delta>0$, there exists $v \in \Cu(C/I)$ such that
\[
lv \leq [1], \andSep
\widehat{[1]} \leq l\widehat{v}+\delta\widehat{[1]}.
\]
Applied for $l=k$ and $\delta=\tfrac{1}{2N}$, we obtain $v \in \Cu(C/I)$ such that
\[
kv \leq [1],\andSep 
\widehat{[1]} \leq k\widehat{v} + \frac{1}{2N}\widehat{[1]}.
\]
Adding $\tfrac{1}{2N}\widehat{[1]}$ and multiplying by $N$, we get
\[
(1+\frac{1}{2N})N\widehat{[1]} 
\leq Nk\widehat{v} + \widehat{[1]}.
\]
Now we apply \cref{prp:wayBelowLFS}(ii) to find a positive integer $s$ such that
\[
sN[1] 
\leq (s+1)N[1] 
\leq s(Nkv+[1]).
\]
Since $sN[1]$ is compact, we can pick $v' \in \Cu(C/I)$ such that
\[
v' \ll v, \andSep
sN[1] \leq s(Nkv'+[1]).
\]
Using also that $v \leq [1]$ (since $kv \leq [1]$), we can apply \cite[Lemma~2.3~(ii)]{RobRor13Divisibility} to find $e \in (C/I)_+$ such that $v' \leq [e] \leq v$, and so
\[
k[e] \leq [1],\andSep 
sN[1] \leq s(Nk[e]+[1]).
\]

Separately, and using again that $[1]\in\Cu(C/I)$ satisfies condition (L'), we apply \cref{prp:CpctLFD} to the compact element $[1]$ in $\Cu(C/I)$ and for the given $k$ and $\tfrac{1}{2N}$ to obtain $w \in \Cu(C/I)$ such that
\[
w \leq [1],\quad 
(1-\tfrac{1}{2N})k\widehat{w} \leq \widehat{[1]}, \andSep 
[1] \leq kw.
\]
Multiplying by $N$, we get
\[
(N-\tfrac{1}{2})k\widehat{w} \leq N\widehat{[1]}.
\]
Using that $[1]$ is compact and $[1]\leq kw$, we find $w'',w' \in \Cu(C/I)$ such that
\[
w'' \ll w' \ll w, \andSep
[1] \leq kw''.
\] 
First, applying \cref{prp:wayBelowLFS}(ii), we find a positive integer $t$ such that
\[
t(N-1)kw'
\leq (t+1)(N-1)kw' 
\leq tN[1].
\]
Then, arguing as above, we find $f \in (C/I)_+$ with $w'' \leq [f] \leq w'$, and thus
\[
[f] \leq [1],\quad 
t(N-1)k[f] \leq tN[1], \andSep
[1] \leq k[f].
\]

\medskip

Let $\bar{e},\bar{f}\in C_+$ be lifts of $e,f$ respectively. Then, there exists $z\in \Cu(I)$ such that 
\[
k[\bar{e}] \leq [1]+z,\quad 
sN[1] \leq s(Nk[\bar{e}]+[1])+z,
\]
and
\[
[\bar{f}] \leq [1]+z,\quad 
t(N-1)k[\bar{f}] \leq tN[1]+z,\quad
[1] \leq k[\bar{f}]+z
\]
in $\Cu (C)$.
	
Set $b=a\bar{e}$ and $c=a\bar{f}$ in $A_{\mathcal{U}}$. 
Using that every element in $C$ commutes with $a$ (in particular, those elements applying the Cuntz subequivalences above) and that every element in $I$ is orthogonal to $a$, one obtains
\[
k[b]\leq [a],\quad 
sN[a] \leq s(Nk[b]+[a]),
\]
and
\[
[c] \leq [a],\quad 
t(N-1)k[c] \leq tN[a],\quad
[a] \leq k[c]
\]
in $\Cu (A_{\mathcal{U}})$.

A finite collection of Cuntz subequivalences in $A_{\mathcal{U}}$ can be `lifted' to $A$ simultaneously up to a cut-down (see, for example, the proof of \cite[Theorem~10.5]{AntPerRobThi24TracesUltra}). Thus, using that $x' \ll x=[a]$, we find $y,z \in \Cu(A)$ such that
\[
ky \leq x,\quad 
sNx' \leq s(Nky+x),
\]
and 
\[
z \leq x,\quad 
t(N-1)kz \leq tNx,\quad
x' \leq kz
\]
in $\Cu(A)$.
Passing to functionals, we can remove the multiplication by $s$ and $t$, and we obtain $N\widehat{x'} \leq Nk\widehat{y} + \widehat{x}$ and $(N-1)k\widehat{z} \leq N\widehat{x}$, which shows that $y$ and $z$ have the desired properties.
\end{proof}

It was proved in \cite[Theorem~3.8]{AntPerRobThi22CuntzSR1} that, if $A$ is a separable \ca{} of stable rank one, then $\Cu(A)$ is an inf-semilattice ordered semigroup. 
That is, for each $x,y\in \Cu(A)$ their greatest lower bound $x\wedge y$ exists in $\Cu(A)$ and addition is distributive over the meet operation:
\[
(x\wedge y)+z=(x+z) \wedge (y+z)
\]
for all $x,y,z\in \Cu(A)$. 
This will be used below.

\begin{prp}
\label{prp:FD-SR1} 	
Every nowhere scattered \ca{} with stable rank one is functionally divisible.
\end{prp}
\begin{proof}
Let us first assume that every separable, nowhere scattered \ca{} with stable rank one is functionally divisible, and let $A$ be an arbitrary nowhere scattered \ca{} with stable rank one.
To show that $A$ is functionally divisible, let $x' \ll x$ in $\Cu(A)$, let $k \in \NN$ with $k \geq 1$ and let $\varepsilon \in (0,1)$ be given.
We need to find $y,z \in \Cu(A)$ such that
\[
ky \leq x, \quad
(1-\varepsilon)\widehat{x'} \leq k\widehat{y}, \quad
z \leq x, \quad 
(1-\varepsilon)k\widehat{z} \leq \widehat{x}, \andSep 
x' \leq kz.
\]
Using that stable rank one and nowhere scatteredness satisfy the L{\"o}wenheim-Skolem condition (\cite[Proposition~4.11]{ThiVil24NowhereScattered}) and applying \cite[Proposition~6.1]{ThiVil21DimCu2}, we can find a separable sub-\ca{} $B \subseteq A$ such that $B$ is nowhere scattered, has stable rank one, and such that the inclusion $B \to A$ induces an order-embedding $\Cu(B) \to \Cu(A)$ whose image contains $x'$ and $x$.
Viewing $x'$ and $x$ as elements in $\Cu(B)$, and using that $B$ is functionally divisible, we find $y$ and $z$ with the desired properties in $\Cu(B)$, and thus in $\Cu (A)$. (See also \cite[Lemma 9.2]{AntPerRobThi22CuntzSR1}.)

\medskip

By the argument above, we may assume that $A$ is a separable, nowhere scattered \ca{} $A$ with stable rank one, and we need to show that $A$ is functionally divisible. 
We first verify condition~(iii) in \cref{prp:CharLFD}, which then implies that $\Cu(A)$ satisfies condition~(L).

Let $x' \ll x$ in $\Cu(A)$, and let $\varepsilon>0$.
We need to find $y \in \Cu(A)$ such that
\[
2y \leq x, \andSep
(1-\varepsilon)\widehat{x'} \leq 2 \widehat{y}.
\]

Using the solution to the rank problem (see Theorem \ref{thm:realization-ranks}) 
we can find $u \in \Cu(A)$ such that $\widehat{u} = \tfrac{1}{2} \widehat{x}$.
Set $v := u \wedge x$, which exists by the comments preceding this proposition.~Using \cite[Theorem~6.12]{AntPerRobThi22CuntzSR1} at the second step, we have
\[
\widehat{v}
= \widehat{u \wedge x}
= \widehat{u} \wedge \widehat{x}
= \frac{1}{2} \widehat{x}.
\]
Using that $\Cu(A)$ is $(2,\omega)$-divisible by the results in 
\cite[Section~5]{AntPerRobThi22CuntzSR1} (see also Theorem~7.1 and Proposition~7.3 in \cite{ThiVil23Glimm}), we can apply \cite[Theorem~5.10]{AsaThiVil23arX:RksSoftOps} to find a soft element $w \in \Cu(A)$ such that
\[
w \leq v, \andSep
\widehat{w} = \widehat{v}.
\]
By \cref{prp:wayBelowLFS}(i), we have
\[
(1-\varepsilon)\frac{1}{2}\widehat{x'} \ll \frac{1}{2}\widehat{x} = \widehat{w},
\]
which allows us to choose $w'',w' \in \Cu(A)$ such that
\[
w'' \ll w' \ll w, \andSep
(1-\varepsilon)\frac{1}{2}\widehat{x'} \leq \widehat{w''}.
\]
Since $w$ is soft, there exists by \cite[Proposition 4.6]{ThiVil24SoftOps} some $t \in \Cu(A)$ such that
\[
w' + t \ll w, \andSep
w' \ll \infty t.
\]
In particular, $w'+t\leq v\leq x$.

Choose $t' \in \Cu(A)$ such that
\[
t' \ll t, \andSep
w' \ll \infty t'.
\]

Applying \axiomO{5} for $w' + t \leq x$ and $w'' \ll w'$ and $t' \ll t$, we obtain $c \in \Cu(A)$ such that
\[
w'' + c \leq x \leq w' + c, \andSep
t' \ll c.
\]

Let us show that $\tfrac{1}{2} \widehat{x} \leq \widehat{c}$.~Since $w'\ll \infty t'$, there is $m\geq 1$ such that $w'\leq mt'$, and thus $x\leq w'+c\leq mt'+t'=(m+1)t'$.
Therefore, if $\lambda \in F(\Cu(A))$ satisfies $\lambda(x)=\infty$, then $\lambda(c)=\infty \geq \tfrac{1}{2}\lambda(x)$.
On the other hand, if $\lambda \in F(\Cu(A))$ satisfies $\lambda(x)<\infty$, then using that $w'\leq w$ and $\widehat{w} = \tfrac{1}{2}\widehat{x}$, we get
\[
\lambda(c) 
\geq \lambda(x) - \lambda(w') 
\geq \lambda(x) - \frac{1}{2}\lambda(x)
= \frac{1}{2}\lambda(x).
\]

Set $y := w'' \wedge c$.
Using that $(1-\varepsilon)\tfrac{1}{2}\widehat{x'} \leq \widehat{w''}$ and $\tfrac{1}{2} \widehat{x} \leq \widehat{c}$, and using \cite[Theorem~6.12]{AntPerRobThi22CuntzSR1} at the second step, we have
\[
(1-\varepsilon)\tfrac{1}{2}\widehat{x'}
\leq \widehat{w''} \wedge \widehat{c}
= \widehat{w'' \wedge c}
= \widehat{y},
\]
and thus 
\[
2y \leq w'' + c \leq x, \andSep
(1-\varepsilon)\widehat{x'} 
\leq 2 \widehat{y},
\]
as desired.

\medskip

Next, to verify condition~(U) from \cref{dfn:FuncDiv}, let $x' \ll x$ in $\Cu(A)$, let $k \in \NN$ with $k \geq 1$, and let $\varepsilon \in (0,1)$.
We need to find $z \in S$ such that
\[
z \leq x,\quad 
(1-\varepsilon)k\widehat{z} \leq \widehat{x}, \andSep 
x' \leq kz.
\]

Set $s:=(1-\varepsilon)^{-1}\tfrac{1}{k}$.
If $s \geq 1$ (that is, if $(1-\varepsilon)k \leq 1$), then we can use $z :=x$.
Thus, we may assume that $s < 1$.
For $f \in L(F(\Cu(A)))$, let $\alpha(f) \in \Cu(A)$ be as in \cref{thm:realization-ranks}, hence $\widehat{\alpha(f)}=f$. 
Set
\[
z := x \wedge \alpha\left( s\widehat{x} \right).
\]
Then $z \leq x$.
Further, using that $s < 1$, and using \cite[Theorem~6.12]{AntPerRobThi22CuntzSR1}, we have
\[
\widehat{z} 
= \widehat{x} \wedge s\widehat{x}
= s\widehat{x}
= \frac{(1-\varepsilon)^{-1}}{k}\widehat{x}
\]
and thus $(1-\varepsilon)k\widehat{z} = \widehat{x}$.

Finally, let us verify that $x' \leq kz$. 
In fact we will see that $x \leq kz$.
Since $\Cu(A)$ is inf-semilattice ordered, we have
\[
kz = k \big( x \wedge \alpha(s\widehat{x}) \big) = \bigwedge_{i=0}^{k} \big(ix+(k-i)\alpha(s\widehat{x})\big).
\] 
Clearly $x\leq ix+(k-i)\alpha(s\widehat{x})$ when $i>0$. For $i=0$, first note that $ks-1 > 0$ and therefore $\widehat{x} \leq \infty (ks-1)\widehat{x}$, which allows us to apply the  partial additivity of $\alpha$ stated in Theorem \ref{thm:realization-ranks}. Hence, we get
\[
x 
\leq x + \alpha\big( (ks-1)\widehat{x} \big)
= \alpha\big( \widehat{x}+(ks-1)\widehat{x} \big)
= \alpha(ks\widehat{x})
= k \alpha(s\widehat{x}),
\]
and thus $x\leq kz$, as desired. 
\end{proof}

\begin{prp}
\label{prp:FD-Monotrace} 	
Let $A$ be a simple, unital, non-elementary \ca{} with a unique normalized quasitrace.
Then $\Cu(A)$ is functionally divisible.
\end{prp}
\begin{proof}
By assumption, there exists a unique functional $\lambda \colon \Cu(A) \to [0,\infty]$ satisfying $\lambda([1])=1$.
Note that any two elements $x,y \in \Cu(A)$ satisfy $\widehat{x} \leq \widehat{y}$ if and only if $\lambda(x) \leq \lambda(y)$.

We first verify condition~(iii) in \cref{prp:CharLFD}, which then implies that $\Cu(A)$ satisfies condition~(L) in \cref{dfn:FuncDiv}.
Let $x' \ll x$ in $\Cu(A)$, and let $\varepsilon\in(0,1)$.
We need to find $y \in \Cu(A)$ such that
\[
2y \leq x, \andSep
(1-\varepsilon)\lambda(x') \leq 2 \lambda(y).
\]

If $x'=0$, then we can use $y=0$.
Therefore, we may from now on assume that $x' \neq 0$.
Then $\lambda(x') \in (0,\infty)$.
Since $A$ is non-elementary, we can pick $v \in \Cu(A)$ such that $\lambda(v) = \tfrac{1}{2} \lambda(x')$.
We have
\[
(1-\varepsilon)\frac{1}{2}\lambda(x') 
< \frac{1}{2}\lambda(x') 
\leq \lambda(x), \lambda(v).
\]

Using that $\Cu(A)$ satisfies Edwards' condition for $\lambda$ (see Theorem~4.7 and Remark~4.2(3) in \cite{Thi20RksOps}, and \cite{AntPerRobThi21Edwards}), we obtain $w \in \Cu(A)$ such that
\[
(1-\varepsilon)\frac{1}{2}\lambda(x') < \lambda(w), \andSep
w \leq v, x.
\]
Choose $w'',w' \in \Cu(A)$ such that
\[
(1-\varepsilon)\frac{1}{2}\lambda(x')  < \lambda(w''), \andSep
w'' \ll w' \ll w.
\]
Applying \axiomO{5} for $w''\ll w'\leq x$, we obtain $c \in \Cu(A)$ such that
\[
w'' + c \leq x \leq w' + c.
\]

Using that $\lambda(x') \leq  \lambda(x)$ and $\lambda(w') \leq \lambda(w) \leq \lambda(v) = \tfrac{1}{2}\lambda(x')$, we get
\[
\frac{1}{2}\lambda(x') 
= \lambda(x') - \frac{1}{2}\lambda(x') 
\leq \lambda(x) - \lambda(w') 
\leq \lambda(c).
\]

Then
\[
(1-\varepsilon)\frac{1}{2}\lambda(x') < \lambda(w''), \andSep
(1-\varepsilon)\frac{1}{2}\lambda(x') < \frac{1}{2}\lambda(x') \leq \lambda(c).
\]
Applying Edwards' condition again, we obtain $y \in \Cu(A)$ such that
\[
(1-\varepsilon)\frac{1}{2}\lambda(x') < \lambda(y), \andSep
y \leq w'', c.
\]
Then $2y\leq w''+c\leq x$. 
Thus $y$ has the desired properties.

\medskip

Next, to verify condition~(U) of \cref{dfn:FuncDiv}, let $x' \ll x$ in $\Cu(A)$, let $k \in \NN$ with $k \geq 1$, and let $\varepsilon\in(0,1)$.
We need to find $z \in \Cu(A)$ such that
\[
z \leq x,\quad 
(1-\varepsilon)k\lambda(z) \leq \lambda(x), \andSep 
x' \leq kz.
\]
Choose $\delta>0$ such that
\[
1+(2k-1)\delta 
\leq (1-\varepsilon)^{-1}.
\]
Then pick $v'',v',v \in \Cu(A)$ such that
\[
x' \ll v'' \ll v' \ll v \ll x, \andSep
\lambda(v) \leq (1+\delta)\lambda(v'').
\]
We have already verified condition~(L) of \cref{dfn:FuncDiv}.
Applied for $v' \ll v$, for~$k$ and for~$\tfrac{\delta}{2}$, we obtain $y \in \Cu(A)$ such that
\[
ky \leq v, \andSep
\left( 1-\frac{\delta}{2} \right)\lambda(v') \leq k\lambda(y).
\]

By \cref{prp:wayBelowLFS}(i), we have $(1-\delta)\widehat{v''} \ll (1-\tfrac{\delta}{2})\widehat{v'}$ in $\Lsc(F(\Cu(A)))$, which allows us to choose $y' \in \Cu(A)$ such that
\[
y' \ll y, \andSep
(1-\delta)\lambda(v'') \leq k\lambda(y').
\]
Applying \cref{prp:DivO5} for $ky \leq v$ and $y' \ll y$, we obtain $z \in \Cu(A)$ such that
\[
(k-1)y' + z \leq v \leq (k-1)y + z, \andSep
v \leq kz.
\]
Using that $\lambda(y')$ is finite, we get
\begin{align*}
k\lambda(z) 
&\leq k\lambda(v) - (k-1)k\lambda(y') \\
&\leq k(1+\delta)\lambda(v'') - (k-1)(1-\delta)\lambda(v'') \\
&= (1+(2k-1)\delta)\lambda(v'')
\leq (1-\varepsilon)^{-1}\lambda(v''),
\end{align*}
and thus
\[
(1-\varepsilon)k\lambda(z) 
\leq \lambda(v'')
\leq \lambda(x).
\]
We further have 
\[
z \leq v \leq x, \andSep
x' \leq v \leq kz,
\]
which shows that $z$ has the claimed properties.
\end{proof}

\section{Pure C*-algebras}
\label{sec:pure}

This section is devoted to proving \cref{ThmB,ThmIntro:CharPure}. 
We actually show that a \ca{} is pure if, and only if, it has controlled comparison and is functionally divisible; if, and only if, it is $(m,n)$-pure for some $m$ and $n$. 
Building on the results from \cref{sec:divisibility} concerning automatic functional divisibility, we show that a \ca{} with controlled comparison is pure if it is either nowhere scattered with real rank zero or stable rank one, or also if it is simple, unital, non-elementary with a unique normalized quasitrace.

\begin{dfn}
\label{dfn:pure}
Given $m,n \in \NN$, we say that a \CuSgp{} is \emph{$(m,n)$-pure} if it has $m$-comparison	and is $n$-almost divisible.
We say that a \CuSgp{} is \emph{pure} if it is $(0,0)$-pure.
	
A \ca{} is \emph{$(m,n)$-pure} if its Cuntz semigroup is.
Similarly, a \ca{} is \emph{pure} if it is $(0,0)$-pure.
\end{dfn}

\begin{prp}
\label{prp:ZStablePure}
Every $\mathcal{Z}$-stable \ca{} is pure.
Every pure \ca{} has the Global Glimm Property and strict comparison (of positive elements by quasitraces).
\end{prp}
\begin{proof}
Let $A$ be a $\mathcal{Z}$-stable \ca.
Then the (classical, uncomplete) Cuntz semigroup $W(A)$ is almost unperforated by \cite[Theorem~4.5]{Ror04StableRealRankZ}. 
Further, by \cite[Theorem~5.35]{AraPerTom11Cu}, for every $x \in W(A)$ and every $k \geq 1$ there exists $y \in W(A)$ such that $ky \leq x \leq (k+1)y$.
Using that $W(A)$ is order-dense in $\Cu(A)$ (see \cite[Theorem~3.2.8]{AntPerThi18TensorProdCu}), it follows that $\Cu(A)$ is almost unperforated and almost divisible (see also \cref{rmk:n-divisibility}), and thus pure.
	
By definition, every pure \ca{} is almost divisible, and therefore $(2,\omega)$-divisible, which by \cite[Theorem~3.6]{ThiVil23Glimm} is equivalent to the Global Glimm Property.
Further, every pure \ca{} has $0$-comparison, and thus enjoys strict comparison as noted in \cref{rmk:str-comparison}.
\end{proof}

One can ask whether, for  \CuSgp{s}, $(0,0)$-pureness agrees with some sort of tensorial absorption. This was explored in \cite{AntPerThi18TensorProdCu}, together with an extensive analysis of the tensor product in the category $\Cu$.

\begin{thm}[{\cite[Theorem~7.3.11]{AntPerThi18TensorProdCu}}]
\label{prp:PureCuZMult}
A \CuSgp{} $S$ is pure if and only if $S \cong \Cu(\mathcal{Z}) \otimes S$.
\end{thm}

If the Cuntz semigroup of a \ca{} $A$ is isomorphic to that of $\mathcal{Z} \otimes A$, then~$A$ is clearly pure.
In the proof of \cite[Theorem~1.2]{Tom11KRigidSlowDimGrowth}, Toms shows that the converse holds for simple \ca{s}. (See also \cite[Section~7]{AntPerPet18PerfConditionsCu}.) 
The following question is thus pertinent:

\begin{qst}
Let $A$ be a pure \ca.
Is $\Cu(A) \cong \Cu(\mathcal{Z} \otimes A)$?
\end{qst}

We start with an important technical result, which generalizes some of the \CuSgp{} techniques underlying Proposition~6.4 and Theorem~10.5 in \cite{AntPerRobThi24TracesUltra}.

\begin{prp}
\label{prp:TechnicalCore}
Let $(S,\Sigma)$ be a scaled \CuSgp{} satisfying \axiomO{5}. 
Assume that $(S,\Sigma)$ has controlled comparison and is functionally divisible. 
Then $S$ is pure.
\end{prp}
\begin{proof}
We first prove that $S$ is almost unperforated. 
Let $x,y\in S$, $n\in\NN$, and assume that $(n+1)x\leq ny$. 
We must show that $x\leq y$. 
This is equivalent to showing that $x'\leq y$ for any $x'\in S$ such that $x'\ll x$.
	
Therefore, let $x' \in S$ with $x' \ll x$, and pick $x''\in S$ such that $x' \ll x'' \ll x$.  
Since we have $\widehat{x}\leq\tfrac{n}{n+1}\widehat{y}$, we may choose $\gamma,\gamma'$ such that $\tfrac{n}{n+1}<\gamma<\gamma'<1$. 
Using \cref{prp:wayBelowLFS}(i) one gets $\widehat{x''}\ll\gamma\widehat{y}$, which allows us to pick $y'\in S$ such that
\[
\widehat{x''} \ll \gamma\widehat{y'}, \andSep
y' \ll y.
\]
	
Take $y''$ such that $y'\ll y''\ll y$, and choose $d\in\NN$ such that $x'',y''\in\Sigma^{(d)}$. 
Since $(S,\Sigma)$ has controlled comparison, we obtain $N=N(\gamma',d)\in\NN$ such that 
\[
\widehat{v}\leq\gamma'\widehat{w} \quad \text{ implies } \quad v\leq_N w, \quad \text{ for all } v,w\in\Sigma^{(d)}.
\]
	
Pick $\varepsilon>0$ such that $\gamma \leq (1-\varepsilon)^2\gamma'$.
Applying condition~(U) in \cref{dfn:FuncDiv} for $x'\ll x''$ and $N$, we obtain $v \in S$ satisfying
\[
v\leq x'',\quad 
(1-\varepsilon)N\widehat{v} \leq \widehat{x''}, \andSep
x' \leq Nv.
\]
Similarly, applying condition~(L) in \cref{dfn:FuncDiv} for $y' \ll y''$ and $N$, one gets $w \in S$ such that
\[
Nw \leq y'', \andSep
(1-\varepsilon)\widehat{y'} \leq N\widehat{w}.
\]
	
Combining the functional inequalities for $\widehat{v},\widehat{w}$ and using that $\widehat{x''}\leq\gamma \widehat{y'}$ at the second step, we have
\[
(1-\varepsilon)^2 N \widehat{v} 
\leq (1-\varepsilon)\widehat{x''}
\leq (1-\varepsilon)\gamma\widehat{y'}
\leq \gamma N \widehat{w}
\leq (1-\varepsilon)^2 N \gamma'\widehat{w},
\]
which implies that $\widehat{v}\leq\gamma'\widehat{w}$.
	
Note that both $v$ and $w$ belong to $\Sigma^{(d)}$, since $v\leq x'', Nw\leq y''$, and $x'',y''\in\Sigma^{(d)}$. 
It follows that $v\leq_N w$, and thus
\[
x' \leq Nv \leq Nw \leq y''\leq y,
\]
as desired.

\medskip

To show that $S$ is almost divisible, let $x' \ll x$ in $S$, and let $k \in \NN$.
Pick $\varepsilon\in(0,1)$ such that $(1-\varepsilon)^{-1}k < k+1$.
Applying condition~(L) in \cref{dfn:FuncDiv} for $x'' \ll x$, and for~$k$ and $\varepsilon$, we obtain $y \in S$ such that 
\[
ky \leq x, \andSep
(1-\varepsilon)\widehat{x'} \leq k\widehat{y}.
\]
It follows that $x'\leq \infty y$; 
see, for example, \cite[Lemma~6.6(iii)]{AntPerRobThi22CuntzSR1}. 
Since $(1-\varepsilon)^{-1}k < k+1$, we also have that $\lambda (x')<\lambda ((k+1)y)$ for every functional normalized at $y$. Since $S$ is almost unperforated, we obtain $x'\leq (k+1)y$ by \cite[Proposition~6.2]{EllRobSan11Cone}.
\end{proof}

\begin{rmk}
\label{rmk:adapt}
It is possible to adapt the proof of \cref{prp:TechnicalCore} so that instead of~(U) it suffices to assume the condition ($U_{2^m}$) from \cref{rmk:osti} for all $m$.
Namely, the following stronger form of \cref{prp:TechnicalCore} is available:
If a scaled \CuSgp{} $(S,\Sigma)$ satisfies \axiomO{5} and has controlled comparison, and for every $x' \ll x$ in $S$ and every $\varepsilon \in (0,1)$ there exist $y,z \in S$ such that
\[
2y \leq x, \quad
(1-\varepsilon)\widehat{x'} \leq 2\widehat{y}, \quad
z \leq x, \quad 
(1-\varepsilon)2\widehat{z} \leq \widehat{x}, \andSep 
x' \leq 2z,
\]
then $S$ is pure.
\end{rmk}

\begin{thm}
\label{prp:PureMain}
Given a \ca{} $A$, the following statements are equivalent:
\begin{enumerate}[{\rm (i)}]
\item
$A$ is pure, that is, $A$ has strict comparison and is almost divisible.
\item
$A$ is $(m,n)$-pure for some $m,n\in\NN$.
\item
$A$ has controlled comparison and is functionally divisible. 
\end{enumerate}
\end{thm}
\begin{proof}
It is clear that~(i) implies~(ii).
Further, (ii) implies~(iii) since $m$-comparison implies controlled comparison by \cref{prp:m-comparison-CC}, and $n$-almost divisibility implies functional divisibility by \cref{prp:FD-NAlmDiv}.
Finally, (iii) implies~(i) by \cref{prp:TechnicalCore} since the Cuntz semigroup of every \ca{} satisfies \axiomO{5}.
\end{proof}

\begin{rmk}
If follows from \cref{prp:PureMain} (see also \cref{rmk:n-divisibility}) that $A$ is $(m,n)$-pure if, and only if, $\Cu(A)$ has $m$-comparison and is $n$-almost divisible in the sense of \cite[Definition~3.5]{Win12NuclDimZstable}, that is, given $x\in \Cu(A)$ and $k\geq 1$ , there is $y\in\Cu(A)$ such that $ky\leq x\leq (n+1)(k+1)y$.
\end{rmk}

\begin{rmk}
\cref{prp:PureMain} shows that the combination of weak forms of comparison and divisibility implies strong forms of comparison and divisibility.
One may wonder if this already holds individually for comparison and divisibility.
	
For the comparison properties, this is not the case. Indeed, there exist \ca{s} that have controlled comparison (or even $m$-comparison for some $m \geq 1$) that do not have $0$-comparison.
For example, if $X=[0,1]^5$ is the five-dimensional cube, then $C(X)$ has nuclear dimension at most five, and therefore has $5$-comparison by \cite[Theorem~1.3]{Rob11NuclDimComp}, but it does not have $0$-comparison since the radius of comparison is at least one by \cite[Theorem~1.1]{EllNiu13RadiusCompCommutative}.
	
Similarly, we expect that the answer to the following question is positive.
\end{rmk}

\begin{qst}
Does there exist a \ca{} that is functionally divisible (or even $n$-almost divisible for some $n$) but not almost divisible?
\end{qst}

\begin{prp}
\label{prp:CharPureWithFD}	
Let $A$ be a functionally divisible \ca{}.
Then the following statements are equivalent:
\begin{enumerate}[{\rm (i)}]
\item
$A$ is pure.
\item
$A$ has strict comparison (of positive elements by quasitraces).
\item
$A$ has $m$-comparison for some $m$.
\item
$A$ has controlled comparison.
\end{enumerate}
\end{prp}
\begin{proof}
By definition, pureness implies $0$-comparison, which is equivalent to strict comparison by \cref{rmk:str-comparison}.
This shows that (i) implies~(ii), which in turn implies~(iii).
Further, (iii) implies~(iv) by \cref{prp:m-comparison-CC}.
Finally, (iv) implies~(i) by \cref{prp:PureMain}.
\end{proof}

Combining \cref{prp:CharPureWithFD}	with the results proving functional divisibility from \cref{sec:divisibility}, we can now verify pureness in various settings:

\begin{thm}
\label{prp:PureSR1}
Let $A$ be a nowhere scattered \ca{} that has real rank zero or stable rank one.
Assume that $A$ has controlled comparison (for example, $m$-comparison for some~$m$).
Then $A$ is pure.
\end{thm}
\begin{proof}
If $A$ is nowhere scattered and has real rank zero, then~$A$ is almost divisible as noted in \cref{rmk:n-divisibility}.
If $A$ is nowhere scattered and has stable rank one, then~$A$ is functionally divisible by \cref{prp:FD-SR1}.
Now, in both cases the result follows from \cref{prp:CharPureWithFD}.
\end{proof}

\begin{thm}
\label{prp:PureMonotrace}
Let $A$ be a simple, unital, non-elementary \ca{} with a unique normalized quasitrace.
Assume that $A$ has controlled comparison (for example, $m$-comparison for some $m$).
Then $A$ is pure.
\end{thm}
\begin{proof}
By \cref{prp:FD-Monotrace}, $\Cu(A)$ is functionally divisible.
Now, the result follows from \cref{prp:CharPureWithFD}.
\end{proof}

\begin{exa} 
\label{exa:groupCAlg}
Let $G$ be an infinite, discrete group such that the reduced group \ca{} $C^*_\red(G)$ is simple.
Then $C^*_\red(G)$ is simple, unital, non-elementary and has a unique trace by \cite[Corollary~4.3]{BreKalKenOza17CSimpleUniqueTr}.
If $G$ is exact, then quasitraces on $C^*_\red(G)$ are traces by Haagerup's theorem (\cite{Haa14Quasitraces}), and it follows that $C^*_\red(G)$ has a unique normalized quasitrace.
In this case, \cref{prp:PureMonotrace} applies, and it follows that $C^*_\red(G)$ is pure if, and only if, it has controlled comparison.

It has been proved very recently in \cite{AmrGaoKunPat24} that, for $\mathbb{F}_n$, the free group on $n$ generators, the reduced group \ca{} $C^*_\red(\mathbb{F}_n)$ has strict comparison (equivalently, if it is pure). The result in fact holds for a large class of discrete groups.
\end{exa}

For non-exact groups, we ask the following question:

\begin{qst}
Do simple reduced group \ca{s} have a unique normalized quasitrace?
\end{qst}

\begin{exa}
\label{exa:Villadsen}
In his groundbreaking work on regularity properties for simple, nuclear \ca{s}, Villadsen constructed examples of such algebras that fail strict comparison, and consequently are not pure.
The algebras of `first type' from \cite{Vil98SimpleCaPerforation} have stable rank one (but possibly a complicated trace simplex), while the algebras of `second type' from \cite{Vil99SRSimpleCa} have higher stable rank but a unique tracial state (and consequently also a unique quasitracial state).
Applying \cref{prp:FD-SR1,prp:FD-Monotrace}, we see that both types of Villadsen algebras are automatically functionally divisible.
	
We deduce that Villadsen algebras that fail strict comparison also fail the much weaker property of controlled comparison (and in particular, they do not have $m$-comparison for any~$m$).
\end{exa}

\section{C*-algebras with the Global Glimm Property and finite nuclear dimension}
\label{sec:GGP}

In this section, we prove \cref{thmD}. For this, we use that a \ca{} has the Global Glimm Property precisely when its Cuntz semigroup is $(2,\omega)$-divisible;
see \cite[Theorem~3.6]{ThiVil23Glimm}. 
We then verify a particular statement of the non-simple Toms-Winter conjecture for separable, locally subhomogeneous \ca{s} with stable rank one and topological dimension zero, which yields \cref{ThmE}.

\medskip

The next result is a variation of \cite[Proposition~3.2(ii)]{RobTik17NucDimNonSimple} for nowhere scattered \ca{s}. 
This has essentially appeared in the argument of \cite[Proposition~4.1]{Vil23arX:NWSMultCAlg}.
We use $\dimnuc(A)$ to denote the nuclear dimension of $A$.

\begin{lma}
\label{prp:PreDivFromNucdim}
Let $m\in\NN$, and let $A$ be a nowhere scattered \ca{} satisfying $\dimnuc(A)\leq m$. 
Then for every $x' \ll x$ in $\Cu(A)$ and every $k \in \NN$ with $k\geq 1$, there exists~$y\in\Cu(A)$ such that
\[
y\ll x, \andSep
x'\ll ky \ll 2(m+1)x.
\]
\end{lma}
\begin{proof} 
By \cite[Proposition~4.12]{ThiVil24NowhereScattered} and \cite[Proposition~2.3]{WinZac10NuclDim}, $A\otimes\mathcal{K}$ is nowhere scattered and of nuclear dimension at most $m$.
We may thus assume that $A$ is stable.

Choose $a\in A_+$ and $\varepsilon>0$ such that
\[
x' \ll [(a-\varepsilon)_+], \andSep 
[a] \ll x.
\]

Since $A$ is nowhere scattered, the hereditary sub-\ca{} $\overline{aAa}$ has no finite-dimensional representations, by \cite[Theorem 3.1]{ThiVil24NowhereScattered}. It also has nuclear dimension at most $m$ (by \cite[Proposition~2.5]{WinZac10NuclDim}), whence we can apply \cite[Proposition~3.2(ii)]{RobTik17NucDimNonSimple} for $k$ (and for $\overline{aAa}$, $a$ and $\varepsilon$) to obtain $b\in\overline{aAa}_+$ such that
\[
[(a-\varepsilon)_+] \leq k[b] \leq 2(m+1)[a].
\]

Then $y:=[b]\in\Cu(A)$ has the desired properties.
\end{proof}

The next two results are inspired by \cite[Lemma~3.4]{RobTik17NucDimNonSimple}.

\begin{lma}
\label{prp:CreateTwoLargeOrthogonals}
Let $m\in\NN$, and let $A$ be a \ca{} with the Global Glimm Property and with $\dimnuc(A) \leq m$.
Set $L := 14m^2+6$.
Then for every $x'\ll x$ in $\Cu(A)$, there exist $y_0,y_1 \in \Cu(A)$ such that 
\[
y_0+y_1\leq x,\andSep 
x'\ll Ly_0, Ly_1.
\]
\end{lma}
\begin{proof}
Let $x' \ll x$ in $\Cu(A)$. 
We claim that there exists a separable sub-\ca{} $B\subseteq A$ such that
\begin{enumerate}[{\rm (i)}]
\item
$B$ has the Global Glimm Property and $\dimnuc(B)\leq m$;
\item
the induced  inclusion $\Cu (\iota )\colon \Cu (B)\to\Cu (A)$ is an order-embedding;
\item
$x'$ and $x$ are in the image of $\Cu (\iota )$.
\end{enumerate}

Indeed, we know from Proposition~2.3(iii) and~2.6 in \cite{WinZac10NuclDim} that the property `$\dimnuc(\_\,)\leq m$' is separably inheritable. Using the same methods from \cite[Section~5]{ThiVil21DimCu2}, one can also see that the Global Glimm Property satisfies the same condition. Now \cite[Proposition~6.1]{ThiVil21DimCu2} allows us to choose a sub-\ca{} $B$ with the required conditions.

Identifying $\Cu(B)$ with its image in $\Cu(A)$, we view $x',x$ as elements in $\Cu(B)$, and it suffices to find $y_0,y_1 \in \Cu(B)$ with $y_0+y_1\leq x$ and $x'\ll Ly_0, Ly_1$.
Therefore we may assume that $A$ is separable.

\medskip

Choose $x''\in\Cu(A)$ such that $x' \ll x'' \ll x$. Since $A$ has the Global Glimm Property, $\Cu(A)$ is $(2,\omega)$-divisible by \cite[Theorem~3.6]{ThiVil23Glimm}. Applied for $x''$ and $x$, we obtain $c\in\Cu(A)$ such that
\[
2c \leq x, \andSep
x'' \ll \infty c.
\]

Choose $c'\in\Cu(A)$ such that $x'' \leq \infty c'$ and $c'\ll c$. Using \cref{prp:DivO5} with $c'\ll c$ and $2c\leq x$, we find $d\in\Cu (A)$ such that 
\[
c'+d\leq x\leq 2d.
\]

Set $e:=d\wedge\infty c'$, which is possible by \cref{thm:O7}. 
Using at the first step that $x'' \leq \infty c'$ and $x'' \leq x \leq 2d$, we get
\[
x'' \leq (2d)\wedge\infty c' = 2(d\wedge\infty c') = 2e.
\]

Choose $e'',e'\in\Cu(A)$ such that
\[
x' \ll 2e'', \andSep
e'' \ll e' \ll e.
\]

Applying \cref{prp:PreDivFromNucdim} for $k=2m+3$ and $e''\ll e'$, we obtain $f\in\Cu(A)$ such that
\begin{equation}
\label{eq:CreateOrthogonals1}
f\ll e', \andSep
e''\ll (2m+3)f \ll 2(m+1)e'.
\end{equation}

Choose $f'\in\Cu(A)$ such that
\[
e'' \ll (2m+3)f', \andSep
f' \ll f.
\]

Applying \axiomO{5} for $f'\ll f\leq e'$, we obtain $g\in\Cu(A)$ such that
\begin{equation}
\label{eq:CreateOrthogonals2}
f'+g \leq e' \leq f+g.
\end{equation}

Set $y_0:=f'$ and $y_1:=g+c'$. Then
\[
y_0+y_1 
= f'+g+c'
\leq e'+c'
\leq d+c'
\leq x.
\]

Further,
\[
x'
\ll 2e'' 
\ll 2(2m+3)f' 
= (4m+6)y_0
\leq (14m^2+6)y_0
= Ly_0.
\]

Multiplying \eqref{eq:CreateOrthogonals2} by $2m+3$ at the first step, and using  \eqref{eq:CreateOrthogonals1} at the second step, we deduce
\begin{equation}
\label{eq:CreateOrthogonals3}
(2m+3)e'
\leq (2m+3)f + (2m+3)g
\leq (2m+1)e' + (2m+3)g.
\end{equation}

We claim that $\widehat{e'}\leq(2m+3)\widehat{y_1}$. 
To show this, let $\lambda\in F(S)$. 
We distinguish two cases:

First, assume that $\lambda(c')<\infty$. Using that $e'\ll e\leq\infty c'$, it follows that $\lambda(e')<\infty$ which allows, in \eqref{eq:CreateOrthogonals3}, to cancel after applying $\lambda$ to get
\[
\lambda(e') \leq 2\lambda (e')
\leq (2m+3)\lambda(g)
\leq (2m+3)\lambda(y_1).
\]

On the other hand, if $\lambda(c')=\infty$, then 
\[
\lambda(e') 
\leq \infty
= (2m+3)\lambda(c')
\leq (2m+3)\lambda(y_1).
\]

Since $\Cu(A)$ has $m$-comparison (see \cref{rmk:m-comparison}), we get
\[
e' \leq (m+1)(2m+3)y_1.
\]

Using that $x'\ll 2e'$, one has 
\[
x'
\ll 2e'
\leq 2(m+1)(2m+3)y_1
= 2(2m^2+5m+3)y_1
\leq (14m^2+6)y_1
= Ly_1.
\]
This shows that $y_0$ and $y_1$ have the desired properties.
\end{proof}

\begin{lma}
\label{prp:CreateLargeOrthogonals}
Given $m,l\in\NN$, set $L:= (14m^2+6)^l$. Then, for every \ca{} $A$ with the Global Glimm Property and with $\dimnuc(A) \leq m$, and for every $x'\ll x$ in $\Cu(A)$, there exist $y_0,\ldots,y_l\in\Cu(A)$ such that 
\[
y_0+\ldots+y_l\leq x,\andSep 
x'\ll Ly_j
\]
for $j=0,\ldots,l$.
\end{lma}
\begin{proof}
Fix some $m \in \NN$, and let $A$ be a \ca{} that has the Global Glimm Property and with $\dimnuc(A) \leq m$.
Write $L(m,l) := (14m^2+6)^l$ for $l \geq 0$.

By induction over $l$, we verify that the following statement holds:
\begin{enumerate}
\item[($D_l$)]
For every $x' \ll x$ in $\Cu(A)$ there exist $y_0,\ldots,y_l\in\Cu(A)$ such that 
\[
y_0+\ldots+y_l\leq x,\andSep 
x'\ll L(m,l)y_j
\]
for $j=0,\ldots,l$.
\end{enumerate}

To verify ($D_0$), given $x' \ll x$ in $\Cu(A)$ we simply use $y_0 = x$.
Further, ($D_1$) was shown in \cref{prp:CreateTwoLargeOrthogonals}.
Now let $l\in\NN$ with $l \geq 2$, and assume that we have proved~($D_{l-1}$). 
To verify~($D_{l}$), let $x' \ll x$ in $\Cu (A)$. 
Applying ($D_{l-1}$) for $x' \ll x$, we get $z_0,\ldots ,z_{l-1} \in \Cu(A)$ such that 
\[
z_0+\ldots+z_{l-1}\leq x, \andSep
x'\ll L(m,l-1)z_j
\]
for $j=0,\ldots,l-1$. 

Set $y_j:=z_j$ for $j = 0,\ldots,l-2$, and take $z_{l-1}'\in\Cu (A)$ such that
\[
z_{l-1}'\ll z_{l-1}, \andSep 
x'\ll L(m,l-1)z_{l-1}'.
\]

Applying ($D_{1}$) for $z_{l-1}'\ll z_{l-1}$, we find $y_{l-1}, y_l\in \Cu (A)$ satisfying 
\[
y_{l-1}+y_l \leq z_{l-1},\andSep 
z_{l-1}'\ll L(m,1) y_{l-1}, L(m,1) y_l.
\]
Then $y_0+\ldots+y_l \leq x$, and
\[
x' 
\ll L(m,l-1)z_j
\leq L(m,l)z_j
= L(m,l)y_j
\]
for $j=0,\ldots,l-2$ and
\[
x'
\ll L(m,l-1)z_{l-1}'
\ll L(m,l-1)L(m,1) y_j
= L(m,l)y_j
\]
for $j=l-1,l$.
This shows that $y_0,\ldots,y_l$ have the desired properties.
\end{proof}

\begin{prp}
\label{prp:DivFromNucdim}
Given $m \in \NN$, there exists $N=N(m)$ with the following property:
For every \ca{} $A$ with the Global Glimm Property and $\dimnuc(A) \leq m$, the Cuntz semigroup $\Cu(A)$ is $N$-almost divisible.
\end{prp}
\begin{proof}
The proof is inspired by that of \cite[Lemma~3.6]{RobTik17NucDimNonSimple}. 
Let $L=L(m,m)$ be the constant obtained from \cref{prp:CreateLargeOrthogonals}. 
We verify the statement for $N:=3L(m+1)$.

Let $A$ be a \ca{} with the Global Glimm Property and $\dimnuc(A) \leq m$.
To show that $\Cu(A)$ is $N$-almost divisible, let $x'\ll x$ in $\Cu(A)$, and let $k\geq 1$. 
We need to find $y \in \Cu(A)$ such that 
\[
ky\leq x, \andSep x'\leq (k+1)(N+1)y.
\]

Applying \cref{prp:PreDivFromNucdim} for $3kL(m+1)$ and $x'\ll x$, we obtain $c\in\Cu(A)$ such that
\[
c\ll x, \andSep
x'\ll 3kL(m+1)c \ll 2(m+1)x.
\]

Choose $c'\in\Cu(A)$ such that
\[
x'\ll 3kL(m+1)c', \andSep
c'\ll c.
\]

Now choose $x''\in\Cu(A)$ such that
\[
3kL(m+1)c' \ll 2(m+1)x'', \andSep
x''\ll x.
\]

Applying \cref{prp:CreateLargeOrthogonals} for $x''\ll x$, we obtain $y_0,\ldots,y_m$ such that
\[
y_0+\ldots+y_m \leq x, \andSep
x'' \ll Ly_j \quad\text{ for $j=0,\ldots,m$}.
\]

For each $j$, we obtain
\[
3kL(m+1)c' 
\ll 2(m+1)x''
\ll 2L(m+1)y_j.
\]
and thus $kc'<_s y_j$.

As noted in \cref{rmk:m-comparison}, $A$ has $m$-comparison, and we get
\[
kc'\leq y_0+\ldots+y_m \leq x.
\]

Further,
\[
x' \ll 3kL(m+1)c'
\leq 3(k+1)L(m+1)c'
= (k+1)Nc'\leq (k+1)(N+1)c',
\]
which shows that $y=c'$ has the desired properties.
\end{proof}

\begin{thm}
\label{prp:GGP-finNucDim-pure}
Every \ca{} with the Global Glimm Property and finite nuclear dimension is pure.
\end{thm}
\begin{proof}
Let $m \in \NN$, and let $A$ be a \ca{} with the Global Glimm Property and $\dimnuc(A) \leq m$.
Then $A$ has $m$-comparison, as noted in \cref{rmk:m-comparison}.
Further, by \cref{prp:DivFromNucdim}, we know that $A$ is $n$-almost divisible for some~$n$. 
Hence, $A$ is $(m,n)$-pure, and the result now follows from \cref{prp:PureMain}.
\end{proof}

\begin{thm}
\label{prp:NonsimpleTW}
Let $A$ be a separable, locally subhomogeneous \ca{} that has stable rank one and topological dimension zero.
Then the following statements are equivalent:
\begin{enumerate}
\item[(1)]
$A$ has the Global Glimm Property and finite nuclear dimension.
\item[(2)]
$A$ is $\mathcal{Z}$-stable.
\item[(3a)]
$A$ is pure.
\item[(3b)]
$A$ is nowhere scattered and has strict comparison of positive elements.
\end{enumerate}
\end{thm}
\begin{proof}
By \cref{prp:GGP-finNucDim-pure}, (1) implies~(3a).
By \cref{prp:ZStablePure}, (2) implies~(3a), and~(3a) implies~(3b).
By \cref{prp:PureSR1}, (3b) implies~(3a).

Every locally subhomogeneous \ca{} has locally finite nuclear dimension by  \cite{NgWin06NoteSH}.
Since locally subhomogeneous \ca{s} have no simple, purely infinite ideal-quotients, it follows from \cite[Theorem~7.10]{RobTik17NucDimNonSimple} that (3a) implies~(2).

Finally, by \cite[Theorem~A]{EllNiuSanTik20drASH}, every $\mathcal{Z}$-stable, locally subhomogeneous \ca{} has decomposition rank at most two, and therefore finite nuclear dimension.
Further, by \cref{prp:ZStablePure}, $\mathcal{Z}$-stability implies the Global Glimm Property.
This shows that~(2) implies~(1).
\end{proof}



\providecommand{\etalchar}[1]{$^{#1}$}
\providecommand{\href}[2]{#2}

\end{document}